\documentclass[11pt]{article}

\usepackage{epsfig,epsf,fancybox}
\usepackage{amsmath}
\usepackage{mathrsfs}
\usepackage{amssymb}
\usepackage{graphicx}
\usepackage{color}
\usepackage{multirow}
\usepackage{paralist}
\usepackage{verbatim}
\usepackage{galois}
\usepackage{algorithmic}
\usepackage{boxedminipage}
\usepackage{booktabs}
\usepackage{accents}
\usepackage{stmaryrd}
\usepackage{subfig}
\usepackage{algorithm,float}
\usepackage{lipsum}

\newtheorem{theorem}{Theorem}[section]

\newtheorem{lemma}[theorem]{Lemma}
\newtheorem{proposition}[theorem]{Proposition}

\newtheorem{definition}{Definition}[section]

\newtheorem{remark}[theorem]{Remark}

\newcommand{\R}{\mathbb{R}}
\newcommand{\T}{\top}

\DeclareMathOperator*{\argmax}{argmax}

\makeatletter
\@addtoreset{equation}{section}
\makeatother



\begin{document}
	\title{An Optimal High-Order Tensor Method for Convex Optimization}

\author{
Bo JIANG
\thanks{Research Institute for Interdisciplinary Sciences, School of Information Management and Engineering, Shanghai University of Finance and Economics, Shanghai 200433, China. Email: isyebojiang@gmail.com. Research of this author was supported in part by NSFC Grants 11771269 and 11831002, and
Program for Innovative Research Team of Shanghai University of Finance and Economics.} \and
Haoyue WANG
\thanks{Operations Research Center, Massachusetts Institute of Technology, Cambridge, MA 02139, USA, (email: haoyuew@mit.edu)} \and
Shuzhong ZHANG
\thanks{Department of Industrial and Systems Engineering, University of Minnesota, Minneapolis, MN 55455, USA (email: zhangs@umn.edu); joint appointment with
Institute of Data and Decision Analytics, The Chinese University of Hong
Kong, Shenzhen, China (email: zhangs@umn.edu).
Research of this author was supported in part by the National Science
Foundation (Grant CMMI-1462408). %, and in part by Shenzhen Research Fund under Grant No.\ KQTD2015033114415450.
}}

\maketitle

\begin{abstract}

This paper is concerned with finding an {\em optimal}\/ algorithm for minimizing a composite convex objective function. The basic setting is that the objective is the sum of two convex functions: the first function is smooth with up to the $d$-th order derivative information available, and the second function is possibly non-smooth, but its proximal tensor mappings can be computed approximately in an efficient manner. The problem is to find -- in that setting -- the best possible (optimal) iteration complexity for convex optimization. Along that line, for the smooth case (without the second non-smooth part in the objective) Nesterov proposed (\cite{Nesterov-1983-Accelerated}, 1983) an optimal algorithm for the first-order methods ($d=1$) with iteration complexity $O\left( 1 / k^2 \right)$, { while high-order tensor algorithms (using up to general $d$th order tensor information) with iteration complexity $O\left( 1 / k^{d+1} \right)$ were recently established in \cite{Baes-2009, Nesterov-2018}.} In this paper, we propose a new high-order tensor algorithm for the general composite case, with the iteration complexity of $O\left( 1 / k^{(3d+1)/2} \right)$, which matches the lower bound for the $d$-th order methods as established in \cite{Nesterov-2018, Shamir-2017-Oracle}, and hence is optimal. Our approach is based on the {\em Accelerated Hybrid Proximal Extragradient}\/ (A-HPE) framework proposed by Monteiro and Svaiter in \cite{Monteiro-Svaiter-2013}, where a bisection procedure is installed for each A-HPE iteration. At each bisection step a proximal tensor subproblem is approximately solved, and the total number of bisection steps per A-HPE iteration is shown to be bounded by a logarithmic factor in the precision required.

\vspace{1cm}

\noindent {\bf Keywords:} convex optimization; tensor method; acceleration; iteration complexity. %; optimal algorithm.

\vspace{1cm}

\noindent {\bf Mathematics Subject Classification:} 90C60, 90C25, 47H05

\end{abstract}

\newpage

d\section{Introduction}

In this paper, we consider the following composite unconstrained
convex optimization:
\begin{equation}\label{Prob:main}
\min_{x\in \mathbb{R}^n} F(x):=f(x)+h(x),
\end{equation}
where $f$ is differentiable and convex, and $h$ is convex but possibly non-smooth. In this context, we assume that convex tensor (polynomial) proximal mappings regarding $h$ can be approximately computed efficiently.
Given that structure, a fundamental quest is to find an {\it optimal}\/ algorithm that solves the above problem, using the available derivative information of the smooth part $f$.

In case $F(x)=f(x)$, and only the gradient information of $f$ is available, Nesterov  \cite{Nesterov-1983-Accelerated} proposed a gradient-type algorithm, which achieves the overall iteration complexity of $O(1/k^2)$, matching the lower bound on the iteration complexity of this class of solution methods, hence is known to be an {\em optimal}\/ algorithm among all the first-order methods. Since Nesterov's seminal work \cite{Nesterov-1983-Accelerated}, especially in the recent years when the large scale machine learning applications have come under the spotlight, there has been a surge of research effort to extend Nesterov's approach to more general settings; see e.g.~\cite{Beck-2009-Fast, Cotter-2011-Better, Lan-2012-Optimal, Drori-2014-Performance, Shalev-2014-Accelerated}, and/or to incorporate certain adaptive strategies to enhance the practical performances of the acceleration; see e.g.~\cite{Lin-2014-Adaptive, Scheinberg-Goldfarb-Bai-2014, Calatroni-Chambolle-2017}. At the same time, there has also been a considerable research effort to fully understand the underpinning mechanism of the first-order acceleration phenomenon; see e.g.~\cite{Bubeck-2015-Geometric, Su-2016-Differential, Wibisono-2016-Variational, Wilson-2016-Lyapunov}.

When the Hessian information is available, Nesterov \cite{Nesterov-2008-Accelerating} proposed an acceleration scheme for cubic regularized Newton's method, and he showed that the iteration complexity bound improves from $O\left( 1 / k^2 \right)$ to $O\left( 1 / k^3 \right)$. A few years later,
Monteiro and Svaiter \cite{Monteiro-Svaiter-2013} proposed a totally
different acceleration scheme, which they termed as {\em Accelerated Hybrid Proximal Extragradient Method}\/ (A-HPE) framework, and they proved that if the second-order information is incorporated into the A-HPE framework then the corresponding accelerated Newton proximal extragradient method
has a superior iteration complexity bound of $O\left( 1 / k^{7/2} \right)$ over $O\left( 1 / k^3 \right)$.
%However, whether a faster second-order method exists is unknown until last
In 2018, Arjevani, Shamir and Shiff \cite{Shamir-2017-Oracle} showed that $O\left( 1 / k^{7/2} \right)$ is actually a lower bound for the oracle complexity of the second-order methods for convex smooth optimization. This shows that the accelerated Newton proximal extragradient method is an optimal second-order method.

%As we can see that by taking advantage of the higher order information, the associated algorithms should have lower iteration complexity.
As evidenced by the special cases $d=1$ and $d=2$, there is a clear tradeoff between the level of derivation information required and the overall iteration complexity improved.
Therefore, a natural and important question arises:
\begin{quote}
	{\bf What is the {\em exact}\/ tradeoff relationship between $d$ and the worst-case iteration complexity?}
\end{quote}
Such question has been in fact raised and addressed in some way in recent works \cite{Birgin-Gardenghi-Martinez-Santos-Toint-2017, Cartis2017, Cartis-Gould-Toint-2018, MARTINEZ2017} in the context of nonconvex optimization. For convex optimization, {the accelerated cubic regularized Newton method was generalized to
	the general high-order case \cite{Baes-2009, Nesterov-2018} with the iteration complexity being $O\left( 1 / k^{d+1} \right)$, where $d$ is the order of derivative information used in the algorithm.
} Jiang, Lin and Zhang \cite{Jiang-2017-Unified} extended Nesterov's approach to accommodate the composite optimization \eqref{Prob:main} and relaxed the requirement on the knowledge of problem parameters such as the Lipschitz constants and the requirement on the exact solutions of the subproblems while maintaining the same iteration bound as in \cite{Baes-2009, Nesterov-2018}.
Along the line of bounding the worst case iteration complexity using up to the $d$-th order derivative information, there have also been significant progresses as well.
Arjevani, Shamir and Shiff~\cite{Shamir-2017-Oracle} showed that the worst case iteration complexity of any algorithm in that setting cannot be better than $O\left( 1 / k^{(3d+1)/2} \right)$. A simplified analysis of the bound can be found in Nesterov~\cite{Nesterov-2018}. So, there was a gap between the achieved iteration bound $O\left( 1 / k^{d+1} \right)$ and the best possible bound of $O\left( 1 / k^{(3d+1)/2} \right)$. Clearly at least one of the two bounds is improvable.
%Unfortunately, such iteration  bound does not match the lower bound for $d$-th order method, which was first obtained in \cite{Shamir-2017-Oracle} and later in \cite{Nesterov-2018}.
In this paper, we aim to
%fill in the gap by presenting an optimal high-order tensor method with an iteration complexity of $O\left( 1 / k^{(3d+1)/2} \right)$.
settle the above theoretical quest by providing a new implementable algorithm whose iteration complexity is precisely $O\left( 1 / k^{(3d+1)/2} \right)$.
As a result, the tradeoff relationship discussed above is pinned down to be exactly $O\left( 1 / k^{(3d+1)/2} \right)$.

Our algorithm is based on the A-HPE framework of Monteiro and Svaiter \cite{Monteiro-Svaiter-2013}, which is presented as Algorithm \ref{A-HPE framework} in this paper. In fact, our algorithm specifies a way to generate an approximate solution through the use of high order derivative information by Taylor expansion. In each iteration, such approximate solution is computed by means of a bisection process. At each bisection step, a regulated convex tensor (polynomial) optimization subproblem is approximately solved. Moreover, we show that, to implement one A-HPE iteration, the number of bisection steps -- each calling to solve a convex tensor subproblems -- is upper bounded by a logarithmic factor in the inverse of the required precision.
Our bisection procedure is similar to the one proposed in \cite{Monteiro-Svaiter-2013} for the case $d=2$; however, a key modification is applied which enables the removal of the so-called ``bracketing stage'' used in \cite{Monteiro-Svaiter-2013}.
{
	 After submitting the first version of the paper, we became aware of two other independent works \cite{Gasnikov-2018-Optimal,Bubeck-2018-Near-optimal} establishing similar iteration bounds as ours, with the main difference being that the focus of \cite{Gasnikov-2018-Optimal,Bubeck-2018-Near-optimal} is on the smooth case: $F(x) = f(x)$, while our method accommodates a composite objective function. The common theoretical development by the three groups was subsequently jointly announced %in an abstract 
	in the form of abstract at Conference on Learning Theory (COLT) \cite{Gasnikov-etal-2019}. It is also worth mentioning that other than the afore-mentioned three papers there are some other related works on high-order optimization methods  \cite{Bullins-2018,Alves-2014-primaldual,Alves-2016-Iteration} based on large-step A-HPE framework.
}

The rest of the paper is organized as follows. In Section \ref{Sec:pre}, we introduce some preliminaries
including the assumptions and the high-order oracle model used throughout this paper. Then we present our optimal tensor method and its iteration complexity analysis in Section \ref{Sec:alg}. The line search subroutine being used in the main procedure of our optimal tensor method is presented and analyzed in
Section \ref{Line search}. Finally, some technical proofs and lemmas are provided in the appendix.

\section{Preliminaries}\label{Sec:pre}

\subsection{Notations}
We denote $\nabla^d f(x)$ to be the  $d$-th order derivative tensor at point $x$ of function $f$ with the $(i_1,...,i_d)$ component given as:
\begin{equation*}
\nabla^d f(x)_{i_1,\dots,i_d}  = \frac{\partial^d f}{\partial x_{i_1}\cdots\partial x_{i_d}} (x),\; \forall 1 \le \; i_1,...,i_d \le n.
\end{equation*}
Given a $d$-th order tensor $\mathcal{T}$ and vectors $z^1,\dots ,z^d \in \mathbb{R}^n $, we denote
\begin{equation*}
\mathcal{T}[z^1,\dots,z^d] := \sum_{i_1,\dots,i_d=1}^{n} \mathcal{T}_{i_1,\dots,i_d} z^{1}_{i_1} \dots z^{d}_{i_d}.
\end{equation*}
The operator norm associated with $\mathcal{T}$ is defined as:
\begin{equation*}
\|\mathcal{T}\| := \max\limits_{\|z^i\|= 1,\,i=1,...,d}  \mathcal{T}[z^1,...,z^d].
\end{equation*}
For given $z^{k+1},\dots,z^{d}$, $\mathcal{T}[z^{k+1},\dots,z^{d}]$ is a $k$-th order tensor with the associated $(i_1,\cdots, i_k)$ component defined as:
$$
\mathcal{T}[z^{k+1},\dots,z^{d}]_{i_1,\cdots, i_k} := \sum_{i_{k+1},\dots,i_{d}=1}^{n} \mathcal{T}_{i_1,\dots,i_{k},i_{k+1},\dots,i_d} z^{k+1}_{i_{k+1}} \dots z^{d}_{i_{d}}
$$
for $1 \le i_1,..., i_k \le n$.
Denote
$$(z^{1}_*,\dots,z^{k}_*) := \argmax\limits_{\|y^i\| = 1,\, i=1,...,k} \left(\mathcal{T}[z^{k+1},...,z^d]\right)[y^1,...,y^k].$$
One has
\begin{equation}\label{tensor-norm-bound}
\|\mathcal{T}[z^{k+1},...,z^d] \| = \mathcal{T}[z^{1}_*,\dots,z^{k}_*,z^{k+1},...,z^d] \le \|\mathcal{T} \|  \| z^{k+1}\| \cdots \| z^{d}\|.
\end{equation}
As a matter of convention, for quantities $x$ and $y$, we use the notation $y=\Theta(x)$ to indicate the relation that there are positive constants $a$ and $b$ such that $a x \le  y \le b x$. If $a$ is absent, then we shall indicate the relation as $y=O(x)$.

\subsection{High-Order Oracle Model and Regularized Tensor Approximation}
In this paper, we consider the following high-order oracle model and the algorithm we are going to propose is such oracle model.

\smallskip

\begin{center}
	\fbox {\begin{minipage}{5in}
			{\bf $d$-th Order Oracle Model}
			
			\begin{itemize}
				\item $f$ is $d$ times Lipschitz-continuous and differentiable
				with Lipschitz constant $L_d$ for $d$-th order derivative tensor; i.e.
				\begin{equation}\label{Lipschitz-continuous}
				\|\nabla^{d}f(x)-\nabla^{d}f(y) \| \leq L_d \|x-y\|  \ \ \ \forall x,y \in \mathbb{R}^n,
				\end{equation}
				where the left side is the $d$-th order tensor operator norm.
				\item Given any $x$, the oracle returns $f(x), \nabla f(x), \nabla^2 f(x), ... , \nabla^d f(x) $.
				\item At iteration $k$, $x_k$ is generated from a deterministic function $h$ %$\partial h(\cdot)$
				and the oracle's responses at any linear combination of
				$x_1, x_2, ... , x_{k-1}$ and $\nabla^i f(x_1),\nabla^i f(x_2), ...,\nabla^i f(x_{k-1})$, where $1\le i \le d$.
				
			\end{itemize}
			
	\end{minipage}}
\end{center}

Recall that the exact proximal minimization at point $x$ with stepsize $\lambda>0$ is defined as
\begin{equation}\label{exact proximal}
\min\limits_{y\in \mathbb{R}^n} \  f(y) +h(y) +\frac{1}{2 \lambda} \|y-x\|^2 .
\end{equation}
To utilize all the derivative information, we consider the regularized tensor approximation of $f(y)$ at point $x$:
\begin{equation}\label{g_x}
f_x (y) := f(x) + \nabla f(x)^\T (y-x) + \frac{1}{2} \nabla^2 f(x)[y-x]^2 + \cdots + \frac{1}{d!} \nabla^d f(x) [y-x]^d + \frac{M}{(d+1)!}\|y-x\|^{d+1},
\end{equation}
where $M>0$ is the parameter of the high-order regularization term $\|y-x\|^{d+1}$.
Then, by \eqref{Lipschitz-continuous} and the Taylor expansion, we can bound the gap between $f_x(\cdot)$ and $f(\cdot)$ for any $x$ (see Nesterov~\cite{Nesterov-2018}):
\begin{lemma}\label{Fun-Gap-Bound}
	For every $x,y \in \mathbb{R}^n$,
	\begin{equation*}
	\|\nabla f(y)-\nabla f_x(y)\| \leq \frac{L_d+M}{d!} \|y-x\|^d .
	\end{equation*}
\end{lemma}
Therefore, it is natural to consider the tensor approximation of \eqref{exact proximal}:
\begin{equation}\label{inexact proximal}
\min\limits_{y\in \mathbb{R}^n} \  f_x(y) +h(y) +\frac{1}{2 \lambda} \|y-x\|^2 .
\end{equation}
{
	In fact, \eqref{inexact proximal} is the subproblem to be solved in the Optimal Tensor Method that will be introduced later. Note that similar subproblems have appeared in \cite{Monteiro-Svaiter-2013} and \cite{Nesterov-2018}. Specifically, the one used in \cite{Monteiro-Svaiter-2013} corresponds to $d=2$ in \eqref{inexact proximal} without
	the term involving $\|y-x\|^{d+1}$ (i.e., $M=0$), while \cite{Nesterov-2018} uses the subproblem that only minimizes $f_x(y) $ (i.e., without the nonsmooth term $h(y)$ and the quadratic regularization term $\frac{1}{2 \lambda} \|y-x\|^2$).
	In contrast, our above subproblem installs both the high-order and quadratic regularization terms.
}

Note that the unique solution $y$ of \eqref{inexact proximal} is characterized by the following optimality condition:
\begin{equation}\label{inexact proximal opt cond.}
u \in  (\nabla f_x + \partial h)(y),\qquad \lambda u + y -x =0.
\end{equation}
For a scalar $\epsilon \ge 0$, the $\epsilon$-subdifferential of a proper closed convex function $h$ is defined as:
$$
\partial_{\epsilon} h (x) := \{ u \;|\;h(y) \ge h(x) + \langle y-x, u \rangle - \epsilon,
\; \forall\; y \in \mathbb{R}^n  \}.
$$
With the above notion in mind, let us consider the following approximate solution for \eqref{inexact proximal opt cond.} (hence \eqref{inexact proximal}).
\begin{definition}\label{Newton approximate def}
	Given $(\lambda,x) \in \mathbb{R}_{++}\times \mathbb{R}^n$ and $\hat{\sigma}\ge 0$, the triplet $(y,u,\epsilon) \in \mathbb{R}^n\times\mathbb{R}^n\times\mathbb{R}_+ $ is called a $\hat{\sigma}$-approximate solution of (\ref{inexact proximal}) at $(\lambda,x)$ if
	\begin{equation}\label{Appro-Sol}
	u \in (\nabla f_x +\partial_{\epsilon} h)(y) \ \ \mbox{ and } \ \ \|\lambda u+y-x\|^2+2\lambda \epsilon \leq \hat{\sigma}^2\|y-x\|^2.
	\end{equation}
\end{definition}
Obviously, if $(y, u)$ is the solution pair of \eqref{inexact proximal opt cond.}, then $(y, u, 0)$ is a $\hat{\sigma}$-approximate solution of (\ref{inexact proximal}) at $(\lambda,x)$ for any $\hat \sigma \ge 0$. In the rest of our analysis, we assume the availability of a subroutine which, for given $(\lambda,x)$ and $\hat{\sigma}> 0$, returns  a $\hat{\sigma}$-approximate solution $(y,u,\epsilon)$. Let us call this subroutine {\bf ATS} { (Approximate Tensor Subroutine).
	Different from \cite{Nesterov-2018}, where a similar subproblem as \eqref{inexact proximal} without a possible nonsmooth function $h(\cdot)$ and regularization term $\frac{1}{2 \lambda}\|y - x\|^2$ is exactly solved, we only assume an approximate solution in the form of 
	\eqref{Appro-Sol} is available and no further assumption on $h(\cdot)$ is required. Note that the possibly nonsmooth function $h(\cdot)$ can be viewed as a fixed parameter in {\bf ATS}. Once $h(\cdot)$ is given,  
	{\bf ATS} could be called in each step of the bisection search, which itself is a subroutine in the main procedure of our algorithm.
}

%In the case $F(x)=f(x)$, {\bf ATS} is invoked in every iteration of the algorithm proposed by Nesterov \cite{Nesterov-2018}. In this paper, we consider the general composite case $F(x)=f(x)+h(x)$, and a proximal version of {\bf ATS} is called in each step of the bisection search, which itself is a subroutine in the main procedure of our algorithm.

\section{The Optimal Tensor Method} \label{Sec:alg}

\subsection{The tensor algorithm and its iteration complexity}
Our bid to the optimal tensor algorithm is based on the so-called {\em Accelerated Hybrid Proximal Extragradient}\/ (A-HPE) framework proposed by Monteiro and Svaiter \cite{Monteiro-Svaiter-2013} for problem \eqref{Prob:main}, whose main steps can be schematically sketched below:

\begin{algorithm}[H]
	\caption{A-HPE framework}
	\label{A-HPE framework}
	\begin{algorithmic}
		\STATE \textbf{STEP 1.} Let $x_0,y_0 \in \mathbb{R}^n$, $0< \sigma<1 $ be given, and set $A_0=0$ and $k=0$.\vspace{0.2cm}
		\STATE \textbf{STEP 2.} If $0 \in \partial F(y_k)$, then \textbf{STOP}. \vspace{0.2cm}
		\STATE \textbf{STEP 3.} Otherwise, find $\lambda_{k+1}>0$ and a triplet $(\tilde{y}_{k+1},v_{k+1},\epsilon_{k+1})$ such that
		\begin{equation}\label{triple relation 1}
		v_{k+1} \in \partial_{\epsilon_{k+1}} F(\tilde{y}_{k+1}),
		\end{equation}
		\begin{equation}\label{triple relation 2}
		\|\lambda_{k+1} v_{k+1} +\tilde{y}_{k+1} -\tilde{x}_k \|^2 +2 \lambda_{k+1}\epsilon_{k+1}
		\leq \sigma^2 \|\tilde{y}_{k+1} -\tilde{x}_k\|^2
		\end{equation}
		\STATE where
		\begin{eqnarray}\label{tilde x}
		\tilde{x}_k &=&\frac{A_k}{A_k+a_{k+1}} y_k + \frac{a_{k+1}}{A_k + a_{k+1} } x_k, \nonumber\\
		a_{k+1} &=& \frac{\lambda_{k+1}+\sqrt{\lambda_{k+1}^2 + 4 \lambda_{k+1} A_k}}{2} \label{a_k+1}. \nonumber
		\end{eqnarray}
		
		\STATE \textbf{STEP 4.} Choose $y_{k+1}$ such that $F(y_{k+1}) \leq F(\tilde{y}_{k+1})$ and let
		\begin{eqnarray}\label{A_k+1}
		A_{k+1} &= &A_k + a_{k+1}, \nonumber \\
		x_{k+1}& =& x_k -a_{k+1} v_{k+1}  \label{x_k+1}. \nonumber
		\end{eqnarray}
		
		\STATE \textbf{STEP 5.} Set $k\leftarrow k+1$, and go to STEP 2.
	\end{algorithmic}
\end{algorithm}

{ 
	Note that in STEP 2, the stopping condition is $0 \in \partial F(y_k)$. However, in practice, the condition is replaced by an approximate version of it (see Algorithm \ref{Alg:Opt-High-Order} below).
}
In the following, we quote some technical results derived in \cite{Monteiro-Svaiter-2013} for A-HPE. Since our proposed algorithm is within that framework, the results in Lemma~\ref{convergence speed_general} hold true for our method as well, and they will be used in the subsequent analysis.

\begin{lemma}\label{convergence speed_general}
	Suppose the sequence $\{x_k, y_k, \tilde{x}_k, \tilde{y}_k
	\}$	is genernated from Algorithm \ref{A-HPE framework}.
	Let {$x_*$} be the projection of $x_0$ onto the set of optimal value points $X_*$, $F_*$ be the optimal value, and $D$ be the distance from $x_0$ to $X_*$. Then for any 	integer $k \geq 1$, it holds that (Theorem 3.6 in \cite{Monteiro-Svaiter-2013}),
	\begin{equation}\label{Bound: Ak-f-yk-D}
	\frac{1}{2} \|{x_*}-x_k \|^2 + A_k(F(y_k)-F_{*})+\frac{1-\sigma^2}{2}\sum^{k}_{j=1} \frac{A_j}{\lambda_j} \|\tilde{y}_j-\tilde{x}_{j-1} \|^2 \leq \frac{1}{2}D^2.
	\end{equation}
	Therefore,
	\begin{equation}\label{Bound: sum-Ak}
	\sum^{k}_{j=1} \frac{A_j}{\lambda_j} \|\tilde{y}_j-\tilde{x}_{j-1} \|^2 \leq \frac{D^2}{1-\sigma^2} .
	\end{equation}
	Furthermore, $A_k$ and $\lambda_k$ has the following relation (Lemma 3.7 in \cite{Monteiro-Svaiter-2013})
	\begin{equation}\label{relation of A and lambda}
	A_k \ge \frac{1}{4} \left(\sum_{j=1}^{k} \sqrt{\lambda_j}\right)^2.
	\end{equation}
	{ If $y_k$ is chosen as $\tilde y_k$ for all k,} the distance between $y_k$ and {$x_*$} can be bounded as follows (Theorem 3.10 in \cite{Monteiro-Svaiter-2013}),
	\begin{equation}\label{bound of yk ineq}
	\|y_k- {x_*}\| \leq \left( \frac{2}{\sqrt{1-\sigma^2}}+1\right)D.
	\end{equation}
\end{lemma}
Now we are ready to propose our optimal tensor method in Algorithm \ref{Alg:Opt-High-Order}.

\begin{algorithm}[!h]
	\caption{The optimal tensor method}
	\label{Alg:Opt-High-Order}
	\begin{algorithmic}
		\STATE \textbf{STEP 1.}  Let $x_0=y_0 \in \mathbb{R}^n$, $v_0 \in \partial f(x_0) $, $\epsilon_0=0$, $k =0$ and set $0 < \bar \epsilon, \bar \rho < 1$, $M \ge L_d$. Let $\hat{\sigma} \ge 0$, $0 < \sigma_l < \sigma_u < 1$ such that
		$\sigma:=\hat{\sigma}+\sigma_u<1$ and $\sigma_l(1+\hat{\sigma})^{d-1}<\sigma_u(1-\hat{\sigma})^{d-1}$.
		
		\vspace{0.2cm}
		
		\STATE \textbf{STEP 2.}
		If {$\|v_{k}\| \leq \bar{\rho} $ and $\epsilon_{k} \leq \bar{\epsilon}$}, then \textbf{STOP}. Else, go to STEP 3.
		
		\vspace{0.2cm}
		
		\STATE \textbf{STEP 3.} Find $\lambda_{k+1}$ and a $\hat{\sigma}$-approximate solution\\
		$(y_{k+1},u_{k+1},\epsilon_{k+1}) \in \mathbb{R}^n\times\mathbb{R}^n\times\mathbb{R}_+$ of
		(\ref{inexact proximal}) at $(\lambda_{k+1},\tilde{x}_k)$ such that either
		\begin{equation}\label{two-side condition}
		\frac{d! \sigma_l}{L_d + M } \leq \lambda_{k+1}\|y_{k+1}-\tilde{x}_{k} \|^{d-1} \leq  \frac{d! \sigma_u}{L_d + M }
		\end{equation}
		or $\| \nabla f(y_{k+1}) + u_{k+1} - \nabla f_{\tilde{x}_k}(y_{k+1}) \| \le \bar \rho$ {and $\epsilon_{k+1} \le \bar{\epsilon}$ hold},
		where
		\begin{equation}\label{rep:x-tilde}
		\tilde{x}_k=\frac{A_k}{A_k+a_{k+1}} y_{k} + \frac{a_{k+1}}{A_k+a_{k+1}} x_k
		\end{equation}
		and
		\begin{equation}\label{rep:a-k}
		a_{k+1}=\frac{\lambda_{k+1}+\sqrt{\lambda_{k+1}^2 +4\lambda_{k+1} A_{k}} }{2} .
		\end{equation}
		{
			(Note that $\lambda_{k+1}$ appears in both \eqref{two-side condition} and \eqref{rep:a-k}, and seeking a proper $\lambda_{k+1}$ requires a bisection procedure, to be called Algorithm \ref{alg:bisection} in Section \ref{Line search}.)
		}
		\STATE \textbf{STEP 4.} Let
		\begin{eqnarray}
		v_{k+1}&=&\nabla f(y_{k+1}) + u_{k+1} - \nabla f_{\tilde{x}_k}(y_{k+1}), \label{vk-Alg2}\\
		A_{k+1} &=& A_k +a_{k+1}, \nonumber \\
		x_{k+1} &=& x_k -a_{k+1} v_{k+1}. \nonumber
		\end{eqnarray}
		\STATE Set $k\leftarrow k+1$ and go to STEP 2.
	\end{algorithmic}
\end{algorithm}

At this point, neither Algorithm \ref{A-HPE framework} nor Algorithm \ref{Alg:Opt-High-Order} has been shown to be implementable. In fact, STEP 3 in both algorithms presented above remain unspecified. { Since $\lambda_{k+1}$ appears in both \eqref{two-side condition} and \eqref{rep:a-k}, }
it is even unclear why such solutions as required by STEP 3 exist at all. {Actually, the double roles played by $\lambda_{k+1}$  in \eqref{two-side condition} and \eqref{rep:a-k} are crucial for the overall $O\left( 1 / k^{(3d+1)/2} \right)$ convergence rate. As a tradeoff, such $\lambda_{k+1}$ is not easy to find}.
In Section \ref{Line search}, we shall { discuss a {\em practical} method to find a proper $\lambda_{k+1}$} (and thus establish a practical implementation of STEP 3 in Algorithm 2) via the
{Approximate Tensor Subroutine ({\bf ATS})}  in combination with a line-search subroutine.

First, let us remark that Algorithm \ref{Alg:Opt-High-Order} is indeed a specialization of A-HPE.
For simplicity, we let $y_{k+1} = \tilde{y}_{k+1}$ in STEP 4 of Algorithm \ref{A-HPE framework}. Because $(y_{k+1},u_{k+1},\epsilon_{k+1})$ is a $\hat{\sigma}$-approximate solution at $(\lambda_{k+1},\tilde{x}_k)$, one has that $u_{k+1} \in (\nabla f_{\tilde{x}_k} +
\partial_{\epsilon_{k+1}}h)(y_{k+1})$, and so we have
\begin{eqnarray*}
	v_{k+1} &\in& \nabla f(y_{k+1}) - \nabla f_{\tilde{x}_k}(y_{k+1}) + (\nabla f_{\tilde{x}_k} +
	\partial_{\epsilon_{k+1}}h)(y_{k+1}) \\
	&=& \nabla f(y_{k+1}) + \partial_{\epsilon_{k+1}}h(y_{k+1}) \subseteq \partial_{\epsilon_{k+1}} (f + h)(y_{k+1})
\end{eqnarray*}
which satisfies \eqref{triple relation 1}. To establish \eqref{triple relation 2}, we need the following proposition.

\begin{proposition}\label{M proposition 7.7}
	Let $(y,u,\epsilon)$ be a $\hat{\sigma}$-approximate solution of (\ref{inexact proximal}) at $(\lambda,{\tilde{x}})$ such that \eqref{two-side condition} holds. Define $v:=\nabla f(y)+u-\nabla f_{\tilde{x}} (y)$. Then,
	\begin{equation}\label{ineq:appr-sol}
	\| \lambda v+y-{\tilde{x}}\|^2+2 \lambda \epsilon \leq \left(\hat{\sigma}+\lambda \frac{L_d+M}{d!} \|y-\tilde{x}\|^{d-1} \right)^2 \|y-\tilde{x}\|^2.
	\end{equation}
	Consequently,
	\begin{equation}\label{ineq:appr-sol2}
	\|\lambda v +{y} -\tilde{x} \|^2 +2 \lambda \epsilon
	\leq \sigma^2 \|{y} -\tilde{x}\|^2 \quad \mbox{with}\quad \sigma=\sigma_u + \hat{\sigma},
	\end{equation}
	{  where $\sigma_u$ is a input paramter in Algorithm \ref{Alg:Opt-High-Order} and also appears in \eqref{two-side condition}.}

	%and
	%\begin{equation}\label{M7.15}
	%  \|v\| \leq \frac{1}{\lambda}\left(1+\hat{\sigma} +\frac{L_d+M}{d!} \lambda \|y-x\|^{d-1}\right)\|y-x\|,\ \ \ \epsilon \leq \frac{\hat{\sigma}^2}{2\lambda}\|y-x\|^2
	%\end{equation}
\end{proposition}

\begin{proof}
	First of all, according to Lemma \ref{Fun-Gap-Bound}, it follows that
	\begin{equation*}
	\lambda \|u-v\| = \lambda \|\nabla f(y)-\nabla f_{\tilde{x}} (y)\| \leq \lambda \frac{L_d+M}{d!} \|y-\tilde{x}\|^d.
	\end{equation*}
	Combining the above inequality with \eqref{Appro-Sol}, one has that
	\begin{eqnarray*}
		&& \|\lambda v+y-\tilde{x}\|^2 +2 \lambda \epsilon \\
		& \leq& \left( \|\lambda u+y-\tilde{x}\| +\lambda\|u-v\| \right)^2 + 2\lambda \epsilon \\
		& = &\left( \|\lambda u+y-\tilde{x}\|^2 +2\lambda \epsilon \right) +2 \lambda \|u-v\|\|\lambda u+y-\tilde{x}\| + \lambda^2 \|u-v\|^2 \\
		& \leq& \hat{\sigma}^2 \|y-\tilde{x}\|^2 + 2\left( \lambda \frac{L_d+M}{d!} \|y-\tilde{x}\|^d \right)\hat{\sigma}\|y-\tilde{x}\|
		+ \left( \lambda \frac{L_d+M}{d!} \|y-\tilde{x}\|^d \right)^2 \\
		& =&\left( \hat{\sigma} +\lambda \frac{L_d+M}{d!} \|y-\tilde{x}\|^{d-1} \right)^2 \|y-\tilde{x}\|^2,
	\end{eqnarray*}
	proving the first inequality. Then, by the { right} hand side of \eqref{two-side condition}, $\lambda \frac{L_d+M}{d!} \|y-\tilde{x}\|^{d-1} \le \sigma_u$, and so the second inequality follows.
\end{proof}

We summarize the above discussion in the theorem below.

\begin{theorem}
	Algorithm \ref{Alg:Opt-High-Order} is a manifestation of the A-HPE framework,
	%	{except it could stop early at STEP 2 due to the relaxed stopping condition and the sequence generated by Algorithm \ref{Alg:Opt-High-Order} is a sequence generated by A-HPE framework}
	and thus the results of Lemma \ref{convergence speed_general} hold for the sequence generated by Algorithm \ref{Alg:Opt-High-Order}.
\end{theorem}

Before addressing the implementation of STEP 3 in Algorithm \ref{Alg:Opt-High-Order}, let us first present the overall iteration complexity of Algorithm \ref{Alg:Opt-High-Order}, assuming STEP 3 could be implemented.
%The key here is to obtain a lower bound on $A_k$, as the following theorem stipulates.

%\begin{theorem}\label{convergence speed_large step}
%	Let $D$ be the distance of $x_0$ to $X_{*}$. Suppose that $\{A_k\}_{k=1}^{\infty}$ is generated from Algorithm~\ref{Alg:Opt-High-Order}. Then for any integer $k \ge 1$, it holds that
%
%{
%	\begin{equation}\label{spd A_k}
%	A_k \ge \left(\frac{1}{2}\right)^{d+1} \frac{d!\sigma_l}{L_d+M} \left(\frac{1-(\sigma_u+\hat \sigma)^2}{D^2}\right)^{\frac{d-1}{2}} \left(\frac{2}{d+1}\right)^{\frac{3d+1}{2}} k^{\frac{3d+1}{2}}.
%	\end{equation}
%}
%\end{theorem}

%The next iteration complexity result readily follows from Theorem \ref{convergence speed_large step}, whose proof will be the subject of Subsection~\ref{Proof-Theorem3.4}.

\begin{theorem}\label{Thm:Iteration-Complexity}
	Let $D$ be the distance of $x_0$ to $X_{*}$. Then, for any integer $k \ge 1$, the iterate $y_k$ generated by Algorithm \ref{Alg:Opt-High-Order} satisfies:
	%	\begin{equation*}
	%	F(y_{k})-F_{*} \leq \frac{D^{\frac{4d}{d+1}}}{2} \left( \frac{d! \sigma_l}{L_d + M }\right)^{-\frac{2}{d+1}}\left(\frac{1}{1-(\hat\sigma^2 + \sigma_u)^2}\right)^{\frac{d-1}{d+1}} \left(\frac{d+1}{2}\right)^{\frac{3d+1}{d+1}} k^{-\frac{3d+1}{2}}.
	%	\end{equation*}
	{
		\begin{equation*}
		F(y_{k})-F_{*} \leq  \left( \frac{d+1}{2} \right)^{\frac{3d+1}{2}} \frac{2^d}{ (1-(\hat \sigma + \sigma_u)^2)^{\frac{d-1}{2}}d!\sigma_l } D^{d+1}(L_d+M)  k^{-\frac{3d+1}{2}}.
		\end{equation*}
		
	}
	
\end{theorem}

%\begin{proof}
%	Combining \eqref{Bound: Ak-f-yk-D} and \eqref{spd A_k} yields that
%{
%	\[
%	F(y_k)-F_{*} \leq \frac{1}{2 A_k}D^2
%	\leq \left( \frac{d+1}{2} \right)^{\frac{3d+1}{2}} \frac{2^d}{ (1-(\hat \sigma + \sigma_u)^2)^{\frac{d-1}{2}}d!\sigma_l } D^{d+1}(L_d+M)  k^{-\frac{3d+1}{2}}.
%	\]
%}
%\end{proof}
The above theorem establishes the $O(1/k^{\frac{3d+1}{2}})$ iteration complexity for Algorithm \ref{Alg:Opt-High-Order}. Since Algorithm~\ref{Alg:Opt-High-Order} falls into the category of the High-Order Oracle Model, whose iteration complexity has a lower bound of $O(1/k^{\frac{3d+1}{2}})$; see Arjevanim, Shamir and Shiff \cite{Shamir-2017-Oracle} and Nesterov \cite{Nesterov-2018}. The worst-case iteration complexity of Algorithm~\ref{Alg:Opt-High-Order} matches this lower bound and it is therefore an {\it optimal}\/ method.

\subsection{Proof of Theorem \ref{Thm:Iteration-Complexity}} \label{Proof-Theorem3.4}
%To establish the lower bound of $\{A_k\}_{k=1}^{\infty}$ in Theorem \ref{convergence speed_large step},
We first provide a recursive bound on $A_k$ as an intermediate step.

\begin{proposition}\label{Prop:Recursive-Bound}Let $D$ be the distance of $x_0$ to $X_{*}$.
	Suppose $\{A_k\}_{k=1}^{\infty}$ is generated from Algorithm \ref{Alg:Opt-High-Order}, then
	\begin{equation}\label{Ak-Recursive-Bound}
	A_k  \ge \frac{1}{4} C^{-\frac{2p}{q}} \left(\sum_{j=1}^k A_j^{\frac{1}{q}}\right)^{2p}
	\end{equation}
	where $q=\frac{3d+1}{d-1}$ , $p=\frac{3d+1}{2d+2}$ and
	$C = \frac{D^2}{(1-({\hat\sigma} + \sigma_u)^2)} \left( \frac{d! \sigma_l}{L_d + M }\right)^{-\frac{2}{d-1}}$.
\end{proposition}
\begin{proof} Suppose $\{x_k,y_k,\tilde{x}_k\}$ is the sequence generated by Algorithm \ref{Alg:Opt-High-Order}. Then, according to \eqref{Bound: sum-Ak} and Proposition \ref{M proposition 7.7}, it holds that
	\begin{equation*}
	\sum^{k}_{j=1} \frac{A_j}{\lambda_j} \|{y}_j-\tilde{x}_{j-1} \|^2 \leq \frac{D^2}{1-({\hat\sigma} + \sigma_u)^2},
	\end{equation*}
	which together with the left hand side of \eqref{two-side condition} implies
	\begin{eqnarray}
	\sum_{j=1}^k \frac{A_j}{\lambda_{j}^{\frac{d+1}{d-1}}} &= &  \sum_{j=1}^k \frac{A_j\|{y}_j-\tilde{x}_{j-1} \|^2}{\lambda_{j}} \cdot \frac{1}{\lambda_{j}^{\frac{2}{d-1}}\|{y}_j-\tilde{x}_{j-1} \|^2} \nonumber \\
	&\le& \frac{D^2}{(1-({\hat\sigma}+\sigma_u)^2)} \left( \frac{d! \sigma_l}{L_d + M }\right)^{-\frac{2}{d-1}} = C. \label{sum bound C}
	\end{eqnarray}
	By the definition of $p$ and $q$, we have $\frac{1}{p}+\frac{1}{q}=1$. Using H\"{o}lder's inequality, together with \eqref{sum bound C}, we have
	\begin{equation*}
	\left(\sum_{j=1}^k \sqrt{\lambda_j}\right)^{\frac{1}{p}}C^{\frac{1}{q}}  \ge \left(\sum_{j=1}^k \sqrt{\lambda_j}\right)^{\frac{1}{p}} \left(\sum_{j=1}^k \frac{A_j}{\lambda_j^{\frac{d+1}{d-1}}}\right)^{\frac{1}{q}} \ge
	\sum_{j=1}^k \lambda_j^{\frac{1}{2p}} \frac{A_j^{\frac{1}{q}}}{\lambda_j^{\frac{d+1}{q(d-1)}}} =
	\sum_{j=1}^k A_j^{\frac{1}{q}}.
	\end{equation*}
	Finally, by \eqref{relation of A and lambda} we obtain
	\begin{equation*}
	A_k \ge \frac{1}{4} \left(\sum_{j=1}^k \sqrt{\lambda_j}\right)^2 \ge \frac{1}{4} C^{-\frac{2p}{q}} \left[\sum_{j=1}^k A_j^{\frac{1}{q}}\right]^{2p}.
	\end{equation*}
	
\end{proof}
%Now let's come back to prove Theorem \ref{convergence speed_large step}.

\noindent {\bf Proof of Theorem \ref{Thm:Iteration-Complexity}.} Let $p$, $q$ and $C$ be defined as in Proposition \ref{Prop:Recursive-Bound}. Construct $\{B_k\}$ such that $B_1=A_1$ and $B_i=T^{\frac{1-(2p/q)^{i-1}}{1-2p/q}} (A_1)^{(2p/q)^{i-1}}$ for $i \ge 2$, where  $T:=\frac{1}{4}(\frac{1}{C})^{\frac{2p}{q}} (\frac{2}{d+1})^{2p}$. Next, we shall apply induction to show that for any $k \ge 1$,
\begin{equation}\label{Ak ge}
A_k \ge B_{i}k^{r_{i}}, \, \forall \, i \ge 1,
\end{equation}
where $r_i=\frac{3d+1}{2} \left[1-(2p/q)^{i-1}\right]$.
When $i = 1$, this is obvious because $A_k \ge A_1 =B_1 k^{r_1}$. Now suppose that for any $k \ge 1$, $A_k \ge B_{i}k^{r_{i}}$ for some $i$. Then, by the induction hypothesis and \eqref{Ak-Recursive-Bound} it holds that
\begin{eqnarray}
A_k &\ge& \frac{1}{4} C^{-\frac{2p}{q}} \left(\sum_{j=1}^k A_j^{\frac{1}{q}}\right)^{2p}
\ge \frac{1}{4} C^{-\frac{2p}{q}} \left(\sum_{j=1}^k (B_i j^{r_i})^{\frac{1}{q}}\right)^{2p} \nonumber \\
&=& \frac{1}{4} \left(\frac{B_i}{C}\right)^{\frac{2p}{q}} \left(\sum_{j=1}^k  j^{\frac{r_i}{q}}\right)^{2p}
\ge \frac{1}{4} \left(\frac{B_i}{C}\right)^{\frac{2p}{q}} \left(\int_{0}^k x^{\frac{r_i}{q}} dx  \right)^{2p} \nonumber \\
& =&   \frac{1}{4} \left(\frac{B_i}{C}\right)^{\frac{2p}{q}}\left(\frac{1}{1+{r_i}/{q}} k^{\frac{r_i}{q} +1}  \right)^{2p}
=    \frac{1}{4} \left(\frac{B_i}{C}\right)^{\frac{2p}{q}} \left(\frac{q}{q+r_i}\right)^{2p} k^{2p\left(\frac{r_i}{q} +1\right)} \nonumber \\
& \ge& \frac{1}{4} \left(\frac{B_i}{C}\right)^{\frac{2p}{q}}\left(\frac{2}{d+1}\right)^{2p} k^{2p\left(\frac{r_i}{q} +1\right)}, \label{Ak-inequ1}
\end{eqnarray}
where the last inequality follows from
\begin{equation*}
\frac{q}{q+r_i}
= \frac{\frac{3d+1}{d-1}}{\frac{3d+1}{d-1} +\frac{3d+1}{2}\left[1-(2p/q)^{i-1}\right]}
=\frac{1}{1+\frac{d-1}{2}\left[1-(2p/q)^{i-1}\right]}
\ge \frac{1}{1+\frac{d-1}{2}}
= \frac{2}{d+1}.
\end{equation*}
%To proceed, we further simplify the two terms: $\frac{1}{4} \left(\frac{B_i}{C}\right)^{\frac{2p}{q}}\left(\frac{2}{d+1}\right)^{2p}$ and $2p \left( \frac{r_i}{q} +1 \right)$ by some straight forward calculation respectively.
Let us further simplify the expression.
First of all, from the definition of $T$ and $B_i$, one observes that
\begin{eqnarray}
\frac{1}{4} \left(\frac{B_i}{C}\right)^{\frac{2p}{q}} \left(\frac{2}{d+1}\right)^{2p}
=B _i^{\frac{2p}{q}} T
& =&\left[ T^{\frac{1-(2p/q)^{i-1}}{1-2p/q}} A_1^{(2p/q)^{i-1}} \right]^{\frac{2p}{q}} T \nonumber \\
& =&T^{\frac{2p/q-(2p/q)^i}{1-2p/q} +1} A_1^{(2p/q)^i} \nonumber \\
& =&T^{\frac{1-(2p/q)^i}{1-2p/q}} A_1^{(2p/q)^i} = B_{i+1}\label{Ak-inequ2}.
\end{eqnarray}
Then, the construction of $q$ and $r_i$ implies that
\begin{eqnarray}
2p \left( \frac{r_i}{q} +1 \right)
& =& \frac{3d+1}{d+1}\left( 1+\frac{\frac{3d+1}{2}(1-(2p/q)^{i-1})}{\frac{3d+1}{d-1}}\right) \nonumber\\
& =& \frac{3d+1}{d+1} \left( 1+ \frac{d-1}{2} \left( 1-(2p/q)^{i-1} \right) \right) \nonumber\\
& =& \frac{3d+1}{d+1} \left( \frac{d+1}{2} - \frac{d-1}{2}(2p/q)^{i-1} \right) \nonumber\\
& =& \frac{3d+1}{2} \left( 1- (2p/q)^{i}\right) =r_{i+1} \label{Ak-inequ3},
\end{eqnarray}
where the second last equality holds true due to the fact that $2p/q = (d-1)/(d+1)$. Now the desired inequality \eqref{Ak ge} follows by combining \eqref{Ak-inequ1}, \eqref{Ak-inequ2} and \eqref{Ak-inequ3}. Observe that $2p/q = (d-1)/(d+1) <1$, and so $\lim\limits_{i \to \infty}B_i = T^{\frac{1}{1 - 2p/q}} = T^{\frac{d+1}{2}}$ and $\lim\limits_{i \to \infty}r_i = \frac{3d + 1}{2}$. Finally, by letting $i\rightarrow \infty$ in \eqref{Ak ge} and using the definition of $C$ in \eqref{sum bound C}, we have
%\begin{eqnarray*}
%	A_k \ge T^{\frac{d+1}{2}} k^{\frac{3d+1}{2}}
%	& =&\frac{1}{4} \left( \frac{1}{C}\right)^{\frac{d-1}{d+1}} \left(\frac{2}{d+1}\right)^{\frac{3d+1}{d+1}} k^{\frac{3d+1}{2}}  \\
%	&  =&  \frac{1}{4} \left( \frac{d! \sigma_l}{L_d + M }\right)^{\frac{2}{d+1}} \left(\frac{1-\sigma^2}{D^2}\right)^{\frac{d-1}{d+1}} \left(\frac{2}{d+1}\right)^{\frac{3d+1}{d+1}} k^{\frac{3d+1}{2}} .
%\end{eqnarray*}
{
	\begin{eqnarray*}
		A_k \ge T^{\frac{d+1}{2}} k^{\frac{3d+1}{2}}
		& =&\left[ \frac{1}{4} \left( \frac{1}{C}\right)^{\frac{d-1}{d+1}} \left(\frac{2}{d+1}\right)^{\frac{3d+1}{d+1}} \right]^{\frac{d+1}{2}} k^{\frac{3d+1}{2}}  \\
		&  =&  \left(\frac{1}{2}\right)^{d+1} \frac{d!\sigma_l}{L_d+M} \left(\frac{1-\sigma^2}{D^2}\right)^{\frac{d-1}{2}} \left(\frac{2}{d+1}\right)^{\frac{3d+1}{2}} k^{\frac{3d+1}{2}}.
	\end{eqnarray*}

	Combining it with \eqref{Bound: Ak-f-yk-D}, we have
	\[
	F(y_k)-F_{*} \leq \frac{1}{2 A_k}D^2
	\leq \left( \frac{d+1}{2} \right)^{\frac{3d+1}{2}} \frac{2^d}{ (1-(\hat \sigma + \sigma_u)^2)^{\frac{d-1}{2}}d!\sigma_l } D^{d+1}(L_d+M)  k^{-\frac{3d+1}{2}}.
	\]

}

\hfill $\Box$ \vskip 0.4cm

\subsection{Comparison with Nesterov's Accelerated Tensor Method}
In Nesterov's accelerated tensor method \cite{Nesterov-2018}, an auxiliary function
\begin{equation}
\psi_k(x) = l_k(x)+M \|x-x_0\|^{d+1}
\end{equation}
with $l_k$ being some linear function, is constructed to satisfy
\begin{eqnarray*}
	\mathcal{R}_{k}^1 &:& \beta_k:= \min\limits_{x} \psi_k(x) -A_k F(y_k) \ge 0 , \nonumber\\
	\mathcal{R}_{k}^2 &:& \psi_k(x) \le A_k F(x) + M \|x-x_0\|^{d+1}, \ \ \ \ \forall x \in \mathbb{R}^n
\end{eqnarray*}
where $A_k = \Theta(k^{d+1})$.
In fact, the function $\psi_k(x)$ serves as a bridge to guarantee the following relation:
\begin{eqnarray}\label{estimation-Nesterove}
A_k F(y_k)\le \min\limits_{x} \psi_k(x) \le \psi_k(x_*) \le
A_k F_* + M \|x_*-x_0\|^{d+1}.
\end{eqnarray}
As a result, $F(y_k) - F_*\le \frac{M}{A_k}\|x_*-x_0\|^{d+1}$ yielding the iteration complexity of $O(1/k^{d+1})$.

In the implementation of high-order A-HPE framework, it is crucial to ensure that condition \eqref{two-side condition} is satisfied. In the remainder of the paper, we shall focus on how to satisfy \eqref{two-side condition} in STEP 3 of Algorithm \ref{Alg:Opt-High-Order}. Our bid is to use bisection on a parameter $\lambda$ (to be introduced later),
while calling an
{ Approximate Tensor Subroutine (\bf ATS)}.
%{\em Approximate Tensor Proximal}\/ {\bf (ATP)} mapping subroutine.
%The cost of having such condiction is that we need to search $\lambda$ in Step 3 of Algorithm \ref{Alg:Opt-High-Order}.
Observe that $\tilde{x}$, which is the point to define $f_{\tilde{x}}(y)$ in \eqref{g_x} to approximate the smooth function $f(y)$, is indeed heavily dependent on $\lambda$. In other words, we need to search for the point where the Taylor expansion \eqref{g_x} is to be computed. This is a key difference between the A-HPE framework and Nesterov's approach \cite{Nesterov-2018}.
Once condition \eqref{two-side condition} is satisfied, then inequality \eqref{Bound: Ak-f-yk-D} would follow, which leads to the following tighter estimation than \eqref{estimation-Nesterove}:
\begin{eqnarray*}
	A_k F(y_k)+\beta_k \le  A_k F_* + \frac{1}{2}\|x_* - x_0 \|^2,
\end{eqnarray*}
as $\beta_k =\frac{1-\sigma^2}{2}\sum^{k}_{j=1} \frac{A_j}{\lambda_j} \|\tilde{y}_j-\tilde{x}_{j-1} \|^2 \ge 0$ is totally missing in \eqref{estimation-Nesterove}. The above inequality also gives an upper bound on $\beta_k$. Together with the lower bound \eqref{bound of yk ineq} this gives a better lower bound on $A_k$, namely $A_k \ge O(k^{\frac{3d+1}{2}})$,  which leads to the optimal iteration complexity presented in Theorem \ref{Thm:Iteration-Complexity}.

\section{A Line Search Subroutine and Its Iteration Complexity} \label{Line search}
After establishing the overall iteration complexity for Algorithm \ref{Alg:Opt-High-Order}, it remains %the critical issue boils down to how to find a $\lambda$ that satisfies the conditions in Step $3$.
to find a way to implement STEP 3 of the algorithm. In this section we discuss how this can be done, from a special case to the general one. The idea is better illustrated by considering the special case. Finally, for the general composite objective function, assuming the tensor proximal mapping regarding $h(x)$ is possible, our approach is based on a line-search procedure for the point on which the Taylor expansion is computed.

%In this section, we first argue the existence of such $\lambda$ under a simplified setting and then provide a line search procedure for implementing Step 3 of Algorithm \ref{Alg:Opt-High-Order}.

\subsection{The Non-Composite Case}

Let us first consider a special case for Algorithm \ref{Alg:Opt-High-Order} where $F(x)=f(x)$ %the non-smooth part $h(x)=0$
in the objective function and $y_{k+1}$ is the exact solution of the following convex tensor proximal point problem:
$$\min\limits_y f_{\tilde{x}_k}(y)+\frac{1}{2\lambda_{k+1}} \|y-\tilde{x}_k\|^2.$$
%and show that there exists a $\lambda_{k+1}$ satisfying
%the alternative condition in Step $3$. We postpone the proof for the general case to the next section, where the complexity for finding such $\lambda_{k+1}$ will be estimated.
We shall discuss how to find $\lambda_{k+1}$ to satisfy the alternative condition in STEP 3 of Algorithm \ref{Alg:Opt-High-Order}.

Note that for fixed $x_k$ and $y_k$, $\tilde{x}_{k}$ and $y_{k+1}$ are uniquely determined by $\lambda_{k+1}$. Therefore the functions $\tilde{x}_{k}(\lambda)$ and $y_{k+1}(\lambda)$ are  continuous with respect to $\lambda$ (where we denote $\lambda_{k+1}$ to be $\lambda$). Next, we show that:
\begin{description}
	\item[(i)] $\lambda \|y_{k+1} (\lambda) - \tilde{x}_k(\lambda)\|^{d-1} \rightarrow 0$, as $\lambda \rightarrow 0$;
	\item[(ii)] Either there exists an increasing sub-sequence $\lambda_j \uparrow \infty$, such that $\lambda_j \|y_{k+1} (\lambda_j) - \tilde{x}_k(\lambda_j)\|^{d-1} \rightarrow \infty$ as
	$j \rightarrow \infty$, or there exists $\hat \lambda$ such that $\| {\nabla f(y_{k+1} (\lambda) )}\|  \le \bar \rho $ for any $\lambda \ge  \hat \lambda$.
\end{description}
Observe that
\begin{eqnarray*}
	& & f_{\tilde{x}_k(\lambda)}(y_{k+1} (\lambda))+ \frac{1}{2\lambda} \|y_{k+1} (\lambda)-\tilde{x}_k(\lambda)\|^2 \\
	&=& \min\limits_{y} f_{\tilde{x}_k(\lambda)}(y)+\frac{1}{2\lambda} \|y-\tilde{x}_k(\lambda)\|^2 \\
	&\le& f_{\tilde{x}_k(\lambda)}(\tilde{x}_k(\lambda)) \\
	&=& f(\tilde{x}_k(\lambda)) < \infty \nonumber,\, \forall \, \lambda>0
\end{eqnarray*}
where $f(\tilde{x}_k(\lambda))$ is bounded, since $\tilde{x}_k(\lambda)$ is a convex combination of $x_k$ and $y_k$. Letting $\lambda \rightarrow 0$ in the above inequality leads to $\|y_{k+1} (\lambda)-\tilde{x}_k(\lambda)\|^2 \rightarrow 0$, which implies $\lambda \|y_{k+1} (\lambda) - \tilde{x}_k(\lambda)\|^{d-1} \rightarrow 0$ as $\lambda \rightarrow 0$, proving {\bf (i)}.

To prove {\bf (ii)}, it suffices to show that if the ``either'' part does not hold, then the ``or'' part must hold. In this case,
there must exist
$ C_1>0$ such that when $\lambda \rightarrow \infty$, $\lambda\|y_{k+1}(\lambda)-\tilde{x}_{k}(\lambda) \|^{d-1} \le C_1$, and thus
$\|y_{k+1}(\lambda)-\tilde{x}_{k}(\lambda) \| \rightarrow 0$.
Moreover, for any $\lambda>0$ the optimality condition is
$$
\nabla f_{\tilde{x}_k(\lambda)} (y_{k+1}(\lambda))  +\frac{1}{\lambda} \left( y_{k+1}(\lambda) - \tilde{x}_k(\lambda) \right) =0 .
$$
Letting $\lambda \rightarrow \infty$ in the above identity yields that $\nabla f_{\tilde{x}_k(\lambda)} (y_{k+1}(\lambda)) \rightarrow 0$. Recall that in this case we have $\|y_{k+1}(\lambda)-\tilde{x}_{k}(\lambda) \| \rightarrow 0$, thus $\nabla f(y_{k+1}(\lambda)) \rightarrow 0 $ proving the ``or'' part.

To summarize, either we have $\lambda \|y_{k+1} (\lambda) - \tilde{x}_k(\lambda)\|^{d-1} \rightarrow 0$ as
$\lambda \rightarrow 0$ and $\lambda_j \|y_{k+1} (\lambda_j) - \tilde{x}_k(\lambda_j)\|^{d-1} \rightarrow \infty$ as
$j \rightarrow \infty$, which guarantees the existence of $\lambda$ to satisfy \eqref{two-side condition} due to the continuity of $\lambda \|y_{k+1} (\lambda) - \tilde{x}_k(\lambda)\|^{d-1}$ on $\lambda$. Or we have a $\lambda_{k+1}$ such that $\| {\nabla f(y_{k+1}(\lambda))  } \|  \le \bar \rho $. In this case, since $h(x)$ is not present, $u_{k+1} = \nabla f_{\tilde{x}_k}(y_{k+1})$ and
$\| \nabla f(y_{k+1}) + u_{k+1} - \nabla f_{\tilde{x}_k}(y_{k+1}) \| = \| \nabla f(y_{k+1})  \|  \le \bar \rho$. Therefore, we have shown that the alternative condition in STEP 3 is actually satisfied.

\subsection{A Bisection Subroutine}

To present the algorithm that computes $\lambda$ satisfing the conditions in STEP 3, we first construct $\beta_{k+1}=\frac{a_{k+1}}{A_k + a_{k+1}}$. From \eqref{rep:a-k}, we can see that $\lambda_{k+1}=\frac{a_{k+1}^2}{A_k+a_{k+1}}$. Therefore, we are able to represent $\lambda_{k+1}$ and $\tilde{x}_k$ by means of $\beta_{k+1}$:
\[
\left\{
\begin{array}{rcl}
\lambda_{k+1} &=&  A_k\frac{\beta_{k+1}^2}{1-\beta_{k+1}} , \\
\tilde{x}_k   &=&  \beta_{k+1} x_{k} + (1-\beta_{k+1}) y_{k}.
\end{array}
\right.
\]
In the $k$-th iteration, we denote
\begin{equation}\label{Expression-lambda}
\lambda(\beta)=A_k\frac{\beta^2}{1-\beta} ,\ \ \beta \in(0,1).
\end{equation}
Its inverse on the domain $\lambda>0$ is
\[
\beta(\lambda) = \frac{ \sqrt{\lambda^2 + 4 \lambda A_k}-\lambda}{2 A_k},
\]
which is monotonically increasing.

%We have following lemma regarding the properties of function $\lambda(\beta)$.
%\begin{lemma}\label{Lemma:func-beta}
%	The function $\lambda(\beta)$ is strictly increasing when $\beta \in (0,1)$. Moreover, $\lambda(\beta)$ has an inverse function $\beta(\lambda)$ with a closed form that $\beta(\lambda) = \frac{ \sqrt{\lambda^2 + 4 \lambda A_k}-\lambda}{2 A_k}$.
%\end{lemma}
%\begin{proof}
%A straightforward calculation shows that
%\begin{equation}\label{Expression-lambda-prime}
%\lambda'(\beta) =A_k \left(\frac{1}{(1-\beta)^2}-1\right) > 0, \; \forall \; \beta \in (0,1).
%\end{equation}	
%Thus $\lambda(\beta)$ is monotonically increasing. Furthermore, the expression of $\lambda(\beta)$ gives that
%$$
%A_k \beta^2 + \lambda \beta  - 1= 0.
%$$
%The above equation has two roots. As $\beta >0$, picking the positive one yields the closed form of $\beta(\lambda)$.
%
%\end{proof}

We shall perform bisection on $\beta$ instead of $\lambda$ in STEP 3 of Algorithm \ref{Alg:Opt-High-Order} to search for $\lambda_{k+1}$. In that way, the initial interval for the bisection is $[0,1]$. (Monteiro and Svaiter \cite{Monteiro-Svaiter-2013} presented a bisection process for their A-HPE algorithm too. However, we can skip what they called the bracketing stage in \cite{Monteiro-Svaiter-2013}).
%Now, we are in position to present our bisection procedure.
%Our bisection procedure is as follows:

\begin{algorithm}[H]
	\caption{Bisection on $\beta$ based on { the subroutine {\bf ATS}}}
	\label{alg:bisection}
	\begin{algorithmic}
		\STATE \textbf{INPUT:} $M \ge L_d$, $\hat{\sigma} \ge 0$, $0 < \sigma_l < \sigma_u < 1$ such that
		$\sigma:=\hat{\sigma}+\sigma_u<1$ and $\sigma_l(1+\hat{\sigma})^{d-1}<\sigma_u(1-\hat{\sigma})^{d-1}$,
		{
			tolerance $\bar \rho>0$ and $\bar \epsilon>0$.}
		\STATE \textbf{STEP 1.} Let
		$\alpha_+=\frac{d!\sigma_u}{L_d+M}$ and $\alpha_-=\frac{d!\sigma_l}{L_d+M}$.

		%\vspace{0.5cm}
		
		% \STATE \textbf{STEP 2.}  (\textbf{Bracketing Stage})
		%        \STATE \quad Set $\alpha_+=\frac{d!\sigma_u}{L_d+M}$ , $\alpha_-=\frac{d!\sigma_u}{L_d+M}$,and compute
		%        \begin{equation}\label{Alg3 lambda+0}
		%          \lambda_+^0:=max \left\{ \alpha_+^{1/d}\left[\frac{1}{\bar{\rho}}(1+\hat{\sigma}+\frac{L_d+M}{d!}\alpha_+)\right]^{1-\frac{1}{d}} ,\ \
		%  \left[\frac{\hat{\sigma}^2 \alpha_+^{\frac{2}{d-1}}}{2\bar{\epsilon}}\right]^{\frac{d-1}{d+1}} \right\}
		%        \end{equation}
		%        \STATE \quad Compute $x_+^0=\tilde{x}_k(\lambda_+^0)$;$(y_+^0,u_+^0,\epsilon_+^0)=blackbox(\lambda_+^0,x_+^0)$,
		%        \STATE \quad and set $v_+^0=\nabla g(y_+^0)-\nabla g_{x_+^0}(y_+^0)+u_+^0$;
		%        \IF{$\|v_+^0\| \leq \bar{\rho}$ and $\epsilon_+^0\leq \bar{\epsilon}$}
		%        \STATE \quad Output $(\lambda,x_{\lambda},y_{\lambda},u_{\lambda},\epsilon_{\lambda})=(\lambda_+^0,x_+^0,y_+^0,u_+^0,\epsilon_+^0)$ and \textbf{STOP}.
		%        \ELSE
		%        \STATE compute $\tau := \frac{1}{\lambda_+^0}\|x_+^0-y_k\| + \frac{2(L_d+M)}{d!} \left[ (2^d+1)\|y_+^0-x_+^0\|^d+2^d \|x_+^0-y_k\|^d \right] $
		%        \begin{equation}\label{Alg3 lambda-0}
		%          \lambda_-^0:=\frac{\alpha_-(1-\hat{\sigma})^d \lambda_+^0}
		%                        { (1+\hat{\sigma})^{d-1} \lambda_+^0\|y_+^0-x_+^0\|^{d-1} +(\lambda_+^0)^2 2^{d-2} max \left\{ 2(\lambda_+^0)^{d-2}\tau^{d-1} ,(d-1)\|x_+^0-y_+^0\|^{d-2}\tau  \right\}   }
		%        \end{equation}
		%        and $x_-^0=\tilde{x}_k(\lambda_-^0)$ , $(y_-^0,u_-^0,\epsilon_-^0)=blackbox(\lambda_-^0,x_-^0)$,and go to STEP 3.
		%        \ENDIF
		
		\vspace{0.2cm}
		
		\STATE \textbf{STEP 2.}  (\textbf{Bisection Setup}) Set $\beta_-=0$, $\beta_+=1$, $\lambda_+ = \lambda(\beta_+)=+\infty$, $\lambda_- = \lambda(\beta_-)$.
		\STATE \textbf{2.a.} Let $\beta=\frac{\beta_- +\beta_+}{2}$ and let
		\begin{equation}\label{formula:x-beta}
		\lambda_{\beta}=\lambda(\beta), \quad x_{\beta}=(1-\beta) y_k+\beta x_k ,
		\end{equation}
		and { use {\bf ATS} to} compute $(y_{\beta},u_{\beta},\epsilon_{\beta}) $
		as a $\hat{\sigma}$-approximate solution at $(\lambda_{\beta},x_{\beta})$,  and $v_{\beta}=\nabla f(y_{\beta}) -\nabla f_{x_{\beta}}(y_{\beta}){ + }u_{\beta}$.
		\STATE \textbf{2.b.}
		\IF{$\|v_{\beta}\| \leq \bar{\rho} $ and $\epsilon_{\beta} \leq \bar{\epsilon}$}
		\STATE  output $(\lambda_{\beta},x_{\beta},y_{\beta},u_{\beta},\epsilon_{\beta})$ and \textbf{STOP}.
		
		\ELSIF{$\lambda_{\beta}\|y_{\beta}-x_{\beta}\|^{d-1} \in [\alpha_-,\alpha_+]$}
		\STATE set $(\beta_{k+1},\tilde{x}_k,y_{k+1},v_{k+1})=(\beta,x_{\beta},y_{\beta},v_{\beta})$ and \textbf{STOP}.
		
		\ELSIF{$\lambda_{\beta}\|y_{\beta}-x_{\beta}\|^{d-1} > \alpha_+$}
		\STATE set $\beta_+ \leftarrow \beta$, and go to STEP 2.a.
		
		%\ELSE
		\ELSIF{ $\lambda_{\beta}\|y_{\beta}-x_{\beta}\|^{d-1} < \alpha_-$}
		
		\STATE set $\beta_- \leftarrow \beta$, and go to STEP 2.a.
		\ENDIF
		
	\end{algorithmic}
\end{algorithm}
{
	We remark that the conditions on
	$\bar \rho$ and $\bar \epsilon$ are only used in the final stage of the algorithm to decide the point that is close to optimum. In
	the implementation, it is reasonable to set a lower precision at the beginning stage of the algorithm.
	Now an upper bound for the overall number of iterations required by Algorithm \ref{alg:bisection} is presented in the following theorem, whose proof will be postponed to the subsequent section.
	\begin{theorem}\label{Thm:complexity-line-search}
		Algorithm \ref{alg:bisection} needs to perform no more than
		\begin{equation}\label{M7.40}
		\Theta \left( \max\{   \log_2(\bar{\epsilon}^{-1})  ,  \log_2(\bar{\rho}^{-1})   \} \right)
		\end{equation}
		bisection steps before reaching $\lambda_{k+1}>0$ and a $\hat{\sigma}$-approximate solution $(y_{k+1},u_{k+1},\epsilon_{k+1})$ at $(\lambda_{k+1},\tilde{x}_k(\lambda_{k+1}))$ satisfying
		$$ \alpha_- \le \lambda_{k+1} \|\tilde{x}_k(\lambda_{k+1})-y_{k+1}\|^{d-1} \le \alpha_+,$$
		or to return $v_{k+1}$ and $\epsilon_{k+1}$ such that $\| v_{k+1} \| \le \bar \rho$ and $ \epsilon_{k+1}  \le \bar \epsilon$.
	\end{theorem}
}

\subsection{The Iteration Complexity Analysis}

%The rest of the section will be presented in the language of Operator Theory. To facilitate the readers, we first review some basic facts in operator theory.
In this subsection, we establish the iteration bound of Algorithm \ref{alg:bisection} {and give a proof for Theorem \ref{Thm:complexity-line-search}}. First, we review some facts for maximal monotone operator.
For a point-to-set operator $T: \R^{n}  \rightrightarrows \R^{n}$, its graph is defined as:
$$
\mbox{Gr}(T) = \{ (z,v) \in \R^{n} \times \R^{n}\; | \; v \in T(z) \},
$$
and the operator $T$ is called {\it monotone}\/ if
$$
\langle v - \tilde{v} , z - \tilde{z} \rangle \ge 0 \quad \forall \; (z,v), \; (\tilde{z},\tilde{v}) \in \mbox{Gr}(T),
$$
and $T$ is {\it maximal monotone}\/ if it is monotone and maximal in the family of monotone operators with respect to the partial order of inclusion. Given a maximal monotone operator $T: \R^{n}  \rightrightarrows \R^{n}$ and a scalar $\epsilon$, the associated $\epsilon$-enlargement $T^{\epsilon}: \R^{n}  \rightrightarrows \R^{n}$ is defined as:
$$
T^{\epsilon}(z) = \{ v \in \R^n \; | \; \langle z - \tilde{z}, v - \tilde{v} \rangle \ge - \epsilon,\, \forall \; \tilde{z} \in \R^n,\; \tilde{v} \in T(\tilde{z}) \},\quad \forall z \in \R^n.
$$
For a convex function $f$, its subdifferential $\partial f$ is monotone if $f$ is a proper function. If $f$ is a proper lower semicontinuous convex function, then $\partial f$ is maximal monotone \cite{Rockafellar-1970}.

Recall that the optimality condition of subproblem \eqref{inexact proximal} is characterized by \eqref{inexact proximal opt cond.}, which is:
$$
0 \in \lambda  (\nabla f_x + \partial h)(y) + y -x = \left( \lambda (\nabla f_x + \partial h) + I \right)(y) -x.
$$
Furthermore, $x$ is optimal to \eqref{Prob:main} if and only if $y = x$. Therefore, it is natural to consider the residual
\begin{equation*}
\varphi(\lambda;x):=\lambda \left\| \left(I + \lambda(\nabla f_x+\partial h)\right)^{-1}(x)-x\right\|
\end{equation*}
for any $\lambda>0,x \in \mathbb{R}^n$. The above residual was adopted in \cite{Monteiro-Svaiter-2013} for the quadratic subproblem. In this paper, to accommodate the high-order information, we consider the following modified residual:
\begin{equation*} %\label{psi}
\psi(\lambda;x):=\lambda \left\| \left(I + \lambda(\nabla f_x+\partial h)\right)^{-1}(x)-x\right\|^{d-1}.
\end{equation*}
We have an immediate property regarding $\psi(\cdot)$.

\begin{proposition}\label{M proposition 7.3}
	Let $x\in \mathbb{R}^n,\lambda>0$ and $\hat{\sigma} \ge 0$. If $(y,u,\epsilon)$ is a $\hat{\sigma}$-approximate solution of (\ref{inexact proximal}) at $(\lambda,x)$, then
	\begin{equation}\label{M7.6}
	\lambda(1-\hat{\sigma})^{d-1} \|y-x\|^{d-1} \leq \psi (\lambda;x) \leq   \lambda(1+\hat{\sigma})^{d-1} \|y-x\|^{d-1} .
	\end{equation}
\end{proposition}

\begin{proof}
	From proposition 7.3 in \cite{Monteiro-Svaiter-2013}, it holds that
	\begin{equation}\label{inexact-exact-gap}
	\left( \lambda(1-\hat{\sigma}) \|y-x\| \right)^{d-1} \leq \varphi^{d-1} (\lambda;x) \leq    \left( \lambda(1+\hat{\sigma}) \|y-x\| \right)^{d-1}.
	\end{equation}
	Notice
	%	\begin{equation*}
	$\varphi^{d-1}(\lambda;x) = \lambda^{d-2}\psi(\lambda;x)$,
	%	\end{equation*}
	and so (\ref{M7.6}) readily follows by combining the above inequalities and identity.
\end{proof}

\begin{lemma}\label{M lemma 7.8}
	Let scalars $\bar{\rho}>0$, $\bar{\epsilon}>0$,   $\hat{\sigma} \ge 0$ and $\alpha>0$ be given  and satisfy $\hat \sigma + \frac{L_d + M}{d !} \alpha := \sigma <1  $. Suppose %$\lambda$ satisfies
	\begin{equation}\label{M7.16}
	\lambda \ge \max \left\{ \alpha^{1/d}\left[\frac{1}{\bar{\rho}}\left(1+\hat{\sigma}+\frac{L_d+M}{d!}\alpha\right)\right]^{1-\frac{1}{d}} ,
	\left(\frac{\sigma^2 \alpha^{\frac{2}{d-1}}}{2\bar{\epsilon}}\right)^{\frac{d-1}{d+1}} \right\},
	\end{equation}
	and $(y,u,\epsilon)$ is a $\hat{\sigma}$-approximate solution of (\ref{inexact proximal}) at $(\lambda,x)$ for some vector $x \in \mathbb{R}^n$. Then, one of the following holds: either {\bf (a)} $\lambda\|y-x\|^{d-1}>\alpha$; or {\bf (b)} the vector $v:=\nabla f(y)-\nabla f_x(y)+u$ satisfies
	\begin{equation}\label{M7.17}
	v\in (\nabla f +(\partial h)^{\epsilon})(y),\ \ \ \|v\|\leq \bar{\rho},\ \ \ \epsilon\leq \bar{\epsilon}.
	\end{equation}
\end{lemma}

\begin{proof}
	Suppose that $\lambda$ satisfies (\ref{M7.16}) but not {\bf (a)}, namely
	\begin{equation}\label{ineq:x-y-power}
	\lambda\|y-x\|^{d-1} \leq \alpha .
	\end{equation}
	In that case, recall that $\partial h_\epsilon$ is the $\epsilon$-subdifferential of $h$ and $(\partial h)^\epsilon$ is the $\epsilon$-enlargement of operator $\partial h$. According to Proposition 3 in \cite{BurachikIusemSvaiter1997}, one has $\partial h_\epsilon(x) \subseteq (\partial h)^\epsilon(x)$ for any $\epsilon \ge 0$ and $x \in \R^n$. Therefore, the inclusion in (\ref{M7.17}) directly follows from Proposition \ref{M proposition 7.7}. Moreover, inequality \eqref{ineq:appr-sol} leads to
	$$
	\lambda \| v \| - \| y -x \| \le  \|\lambda v + y -x \| \le \left(\hat{\sigma} + \lambda \frac{L_d + M}{d !} \|y-x \|^{d-1}\right) \| y - x\| .
	$$
	%    Combining the above inequality with
	Together with
	(\ref{M7.16}) and \eqref{ineq:x-y-power}, the above inequality yields
	\begin{equation*}
	\begin{split}
	\|v\| & \leq \frac{1}{\lambda}\left(1+\hat{\sigma} +\frac{L_d+M}{d!} \lambda \|y-x\|^{d-1}\right)\|y-x\| \\
	& \leq  \frac{1}{\lambda}\left(1+\hat{\sigma} +\frac{L_d+M}{d!} \alpha \right)\left(\frac{\alpha}{\lambda}\right)^{\frac{1}{d-1}} \\
	& \leq \bar{\rho}.
	\end{split}
	\end{equation*}
	On the other hand, inequality \eqref{ineq:appr-sol} also implies that
	$$
	2 \lambda \epsilon \le \left(\hat{\sigma} +\frac{L_d+M}{d!} \lambda \|y-x\|^{d-1}\right)^2\|y-x\|^2
	\le \Big( \hat{\sigma}+\frac{L_d+M}{d!}\alpha  \Big)^2 \|y-x\|^2
	\le \sigma^2 \|y-x\|^2.
	$$		
	Combined with (\ref{M7.16}) and \eqref{ineq:x-y-power} this leads to
	\begin{equation*}
	\epsilon \leq \frac{{\sigma}^2 \|y-x\|^2}{2\lambda}  \leq  \frac{{\sigma}^2 }{2\lambda} \left(\frac{\alpha}{\lambda}\right)^{\frac{2}{d-1}} \leq \bar{\epsilon}.
	\end{equation*}
	Hence, {\bf (b)} must hold in this case.
\end{proof}
In the rest of this section, we simply let $\alpha = \alpha_-$ in
Lemma \ref{M lemma 7.8} and denote
\begin{equation}\label{Def-lambda-bar}
\bar{\lambda}=\max\left\{ \alpha_-^{1/d}\left[\frac{1}{\bar{\rho}}(1+\hat{\sigma}+\frac{L_d+M}{d!}\alpha_-)\right]^{1-\frac{1}{d}} ,\ \
\left(\frac{{\sigma}^2 \alpha_-^{\frac{2}{d-1}}}{2\bar{\epsilon}}\right)^{\frac{d-1}{d+1}}\right\}.
\end{equation}
Lemma \ref{M lemma 7.8} implies that
if $\lambda$ is sufficiently large, then either Algorithm \ref{alg:bisection} stops because \eqref{M7.17} is satisfied, or
$\lambda\| y -x\|^{d-1} \ge \alpha_-$, which achieves half of the bisection goal. {Now we are ready to prove Theorem \ref{Thm:complexity-line-search}.}

\vspace{10cm}

\noindent {\bf Proof of Theorem \ref{Thm:complexity-line-search}.}
Suppose that Algorithm \ref{alg:bisection} has performed $j$ bisection steps before triggering the stopping criteria. %without triggering the stopping criteria.
We aim to show $j \le \Theta \left( \max\{ \log_2(\bar{\epsilon}^{-1})  ,  \log_2(\bar{\rho}^{-1})   \} \right)$.
At that iteration let us denote $x_+ = x_{\beta_+}$, $x_- = x_{\beta_-}$, $y_+ = y_{\beta_+}$ and $y_- = y_{\beta_-}$, and we also have
$\beta_+ - \beta_- = \frac{1}{2^j}$.
%Next, we shall show that $j$ cannot be too large in that case; in particular, $j \le \Theta (\max\{   \log_2(\bar{\epsilon}^{-1})  ,  \log_2(\bar{\rho}^{-1})   \})$.
Denote $\bar \beta = \beta (\bar \lambda)$, where $\bar \lambda$ is as defined in \eqref{Def-lambda-bar}.
%	In the case of $\bar \beta \le \frac{1}{2}$, $\frac{1}{(1 - \bar \beta)} \le 2$. On the other hand, when $\bar \beta > \frac{1}{2}$, the relation of $\beta$ and $\lambda$ in \eqref{Expression-lambda} yeilds that
If $\bar \beta \le \frac{1}{2}$ then $\frac{1}{1 - \bar \beta} \le 2$; if $\bar \beta > \frac{1}{2}$, then \eqref{Expression-lambda} gives
$$	\frac{1}{1 - \bar \beta} = \frac{\bar \lambda}{A_k \bar \beta^2} <  \frac{4 \bar \lambda}{A_k} \le \max \left\{ \Theta \left( ( {\bar \rho}^{-1} )^{\frac{d-1}{d}}\right), \Theta \left(( {\bar \epsilon}^{-1} )^{\frac{d-1}{d+1}}\right) \right\}.$$
Therefore, in the rest of the proof we may assume $j \ge \log_2 (2/(1-\bar{\beta}))$, for otherwise
$	j < \log_2 (2/(1-\bar{\beta})) \le \Theta (\max \{ \log_2 (\bar \rho^{-1}),  \log_2 (\bar \epsilon^{-1}) \})	$
already holds.

Note that the bisection search starts with $\beta_+ = 1$, corresponding to $\lambda_+ = +\infty$ according to \eqref{Expression-lambda} when $\beta_+$ is not updated during the procedure. However, the following lemma tells us that after running Algorithm \ref{alg:bisection} for a number of iterations, $\lambda_+$ will be reduced and upper bounded by some constant depending on $\bar \epsilon$ and $\bar \rho$.
\begin{lemma}\label{lemma:upper-bound-lambda}
	%	When the number of iteration $j$ in Algorithm \ref{alg:bisection} satisfying
	Suppose that Algorithm \ref{alg:bisection} has performed $j$ bisection steps with
	$j \ge \log_2 (2/(1-\bar{\beta}))$, where $\bar{\beta} = \beta (\bar \lambda)$ and $\bar \lambda$ are as defined in \eqref{Def-lambda-bar}. Then we have
	\begin{equation}
	\lambda_+ \le \max\left\{ 9A_k/4,\, 8 \bar \lambda \right\}
	=  \max\left\{ \Theta(\bar \epsilon^{-1}),\, \Theta\left((\bar \rho^{-1})^{\frac{d+1}{d}}\right) \right\}. \label{bound-lambda-plus}
	\end{equation}
\end{lemma}

We shall continue our discussion without disruption here and leave the proof of
Lemma \ref{lemma:upper-bound-lambda} to the appendix. Since Algorithm \ref{alg:bisection} did not stop before iteration $j$, the {bound on $\beta_+$} must have been previously updated, and so
\begin{equation*}
\lambda_+ \|y_{\beta_+} - x_{\beta_+}\|^{d-1} > \alpha_+,
\ \ \  \lambda_- \|y_{\beta_-} - x_{\beta_-}\|^{d-1} < \alpha_-,
\end{equation*}
where $\lambda_+$ is upper bounded due to Lemma \ref{lemma:upper-bound-lambda}.

By Proposition \ref{M proposition 7.3}, we have that
\begin{eqnarray*}
	\psi_+ &:=& \psi(\lambda_+;x_+)\ge \lambda_+(1 - \hat \sigma)^{d-1}\| y_+ - x_+\|^{d-1}>(1-\hat{\sigma})^{d-1}\alpha_+, \\
	\psi_- &:=& \psi(\lambda_-;x_-)\le \lambda_-(1 + \hat \sigma)^{d-1}\| y_- - x_-\|^{d-1}<(1+\hat{\sigma})^{d-1}\alpha_-.
\end{eqnarray*}
Consequently,
\begin{equation}\label{diff of psi lower bound}
\psi_+ - \psi_- > (1-\hat{\sigma})^{d-1}\alpha_+ - (1+\hat{\sigma})^{d-1}\alpha_-.
\end{equation}
The parameters $\alpha_+$ and $\alpha_-$ are pre-specified. Therefore, it suffices to show that $ \psi_+ - \psi_-$ is upper bounded by $\beta_+ - \beta_-$ multiplied by some constant factor and hence the number of bisection search $j$ can be bounded as well.
To this end, denote
\begin{equation}\label{Def-bar-y}
\bar y_+ = \left(I + \lambda_+(\nabla f_{x_+}+\partial h)\right)^{-1}(x_+) \quad \mbox{and} \quad \bar y_- = \left(I + \lambda_-(\nabla f_{x_-}+\partial h)\right)^{-1}(x_-).
\end{equation}
Then, there exist
\begin{eqnarray}
\bar u_+ \in (\nabla f_{x_+}+\partial h)(\bar y_+),& \mbox{\rm s.t.} & \lambda_+ \bar u_+ = x_+ - \bar y_+, \nonumber \\
&& \psi_+ = \lambda_+ \|\bar y_+ - x_+\|^{d-1}=\lambda_+^{d} \|\bar u_+\|^{d-1} \qquad \label{relation+}
\end{eqnarray}
and
\begin{eqnarray}
\bar u_- \in (\nabla f_{x_-}+\partial h)(\bar y_-),& \mbox{\rm s.t.} & \lambda_- \bar u_- = x_- - \bar y_-, \nonumber \\
&& 	\psi_- = \lambda_- \|\bar y_- - x_-\|^{d-1}=\lambda_-^{d} \|\bar u_-\|^{d-1}. \qquad \label{relation-}
\end{eqnarray}
To proceed, we have the following bound on $\lambda_-^2 \|\bar u_+ - \bar u_-\| $ whose proof can be found in the appendix.
\begin{lemma}\label{lemma1} It holds that
	\begin{equation}\label{bound-two-u}
	\lambda_-^2 \|\bar u_+ - \bar u_-\| \leq 2\lambda_-^2 \|\nabla f_{x_+}(\bar y_+)-\nabla f_{x_-}(\bar y_+)\| + | \lambda_+ - \lambda_- | \|\bar y_+ - x_+\|
	+\lambda_- \|x_+ - x_-\| .
	\end{equation}
\end{lemma}	
Note that
\begin{equation}\label{two-power-minus}
\Big|a^{d-1}-b^{d-1}\Big| =\Big| (a-b) (a^{d-2}+a^{d-3}b+\dots +b^{d-2})\Big|\leq (d-1) |a-b| \max\{a, b \}^{d-2},
\end{equation}
for any $a,b>0$. Now combining \eqref{relation+}, \eqref{relation-}, \eqref{bound-two-u} and \eqref{two-power-minus} we have
\begin{eqnarray}\label{diff of psi upper bound}
&&\left| \psi_+ - \psi_-\right|\nonumber \\
&=& \left| \lambda_+^{d} \|\bar u_+\|^{d-1} - \lambda_-^{d} \|\bar u_-\|^{d-1}\right|\nonumber \\
&\leq& \left|\lambda_+^{d}-\lambda_-^{d}\right| \|\bar u_+\|^{d-1} + \left| \|\bar u_+\|^{d-1} - \|\bar u_-\|^{d-1}\right| \lambda_-^d \nonumber \\
&\leq& \left|\lambda_+-\lambda_-\right| d \lambda_+^{d-1} \|\bar u_+\|^{d-1} +
{  \|\bar u_+ - \bar u_-\|(d-1) \max\{\|\bar u_+\|, \|\bar u_-\|\}^{d-2} \lambda_-^d  }\nonumber \\
&=& \left|\lambda_+-\lambda_-\right| d \lambda_+^{d-1} \|\bar u_+\|^{d-1} +
{ (d-1)\|\bar u_+ - \bar u_-\| \max\{ \|\lambda_- \bar u_+\|, \|\lambda_-\bar u_-\|\}^{d-2} \lambda_-^2 } \nonumber \\
&\le& \left|\lambda_+-\lambda_-\right| d \lambda_+^{d-1} \|\bar u_+\|^{d-1} +
{(d-1)\|\bar u_+ - \bar u_-\| \max\{ \|\lambda_+ \bar u_+\|, \|\lambda_-\bar u_-\|\} ^{d-2} \lambda_-^2 } \nonumber \\
&=& d \left|\lambda_+-\lambda_-\right|  \|\bar y_+ - x_+\|^{d-1}
+ { (d-1) \|\bar u_+ - \bar u_-\| \max\{\|x_+ - \bar y_+\|, \|x_- - \bar y_-\|\}^{d-2} \lambda_-^2 } \nonumber \\
&\leq& d \left|\lambda_+-\lambda_-\right|  \|\bar y_+ - x_+\|^{d-1}
+   (d-1) \max\{ \|x_+ - \bar y_+\|, \|x_- - \bar y_-\|\}^{d-2}  \nonumber \\
&& \ \ \times \left( 2\lambda_-^2 \|\nabla f_{x_+}(\bar y_+)-\nabla f_{x_-}(\bar y_+)\| + | \lambda_+ - \lambda_- | \|\bar y_+ - x_+\|
+\lambda_- \|x_+ - x_-\| \right) .
\end{eqnarray}
Next, by applying \eqref{beta-minus-bound}, \eqref{Def-lambda-bar},
Lemma \ref{lemma2}, Lemma \ref{lemma:distance-x-y} and Lemma \ref{lemma4}, we have
\begin{eqnarray*}
	\lambda_- \le \bar \lambda &=&
	\max\Big\{ \Theta\left( \bar \epsilon^{-\frac{d-1}{d+1}} \right),
	\Theta\left( \bar \rho^{-\frac{d-1}{d}} \right)
	\Big\}
	\le
	\max\left\{ \Theta(\bar \epsilon^{-1}),\, \Theta\left(\bar \rho^{-\frac{d+1}{d}}\right) \right\}  , \\
	\lambda_+ - \lambda_-
	&\le& \max\left\{ \Theta\left(\bar \epsilon^{-2}\right), \Theta\left(\bar \rho^{-\frac{2(d+1)}{d}}\right) \right\}(\beta_+ - \beta_-) ,
	\nonumber\\
	\|x_+ - \bar y_+\|
	&\le& \max\left\{ \Theta( \bar \epsilon^{-1}),\, \Theta\left(\bar \rho ^{-\frac{d+1}{d}}\right) \right\} ,
	\nonumber\\
	\|x_- - \bar y_-\|
	&\le& \max\left\{ \Theta\left( \bar \epsilon^{-\frac{d-1}{d+1}} \right),\, \Theta\left(\bar \rho^{-\frac{d-1}{d}}\right) \right\}
	\le
	\max\left\{ \Theta(\bar \epsilon^{-1}),\, \Theta\left(\bar \rho^{-\frac{d+1}{d}}\right) \right\} ,
	\nonumber\\
	\|\nabla f_{x_+}(\bar y_+)-\nabla f_{x_-}(\bar y_+)\|
	&\leq& \max\left\{ \Theta\left(\bar \epsilon^{-d+1}\right), \, \Theta\left( \bar \rho^{-\frac{(d-1)(d+1)}{d}} \right) \right\} (\beta_+ - \beta_-).
\end{eqnarray*}
Combining the bounds above with \eqref{diff of psi upper bound} yields
\begin{eqnarray}
&& \left| \psi_+ - \psi_-\right| \nonumber \\
&\le& d \max\left\{ \Theta\left(\bar \epsilon^{-d-1}\right), \, \Theta\left( \bar \rho^{-\frac{(d+1)^2}{d}} \right) \right\} (\beta_+ - \beta_-) +(d-1)\max\left\{ \Theta\left(\bar \epsilon^{-d+2}\right), \, \Theta\left(\bar \rho^{-\frac{(d+1)(d-2)}{d}} \right) \right\}
\nonumber\\
& & \ \ \times \Bigg(
2\max\left\{ \Theta\left(\bar \epsilon^{-d-1}\right), \, \Theta\left( \bar \rho^{-\frac{(d+1)^2}{d}} \right) \right\}
+  \max\left\{ \Theta\left(\bar \epsilon^{-3}\right), \, \Theta\left(\bar \rho^{-\frac{3(d+1)}{d}} \right) \right\}
\nonumber\\
& & \ \ \ \ \ \  	+  \max\left\{ \Theta\left(\bar \epsilon^{-1}\right), \, \Theta\left(\bar \rho^{-\frac{(d+1)}{d}} \right) \right\}
\Bigg) (\beta_+ - \beta_-)
\nonumber\\
&\le& \max\left\{ \Theta\left(\bar \epsilon^{-2d+1}\right), \, \Theta\left(\bar \rho^{-\frac{(2d-1)(d+1)}{d}} \right) \right\}
(\beta_+ - \beta_-),
\nonumber
\end{eqnarray}
where the last 	inequality is due to $d \ge 2$.
%At that point, $j$ times of bisection steps have been performed; therefore,
Because $\beta_+ - \beta_- = \frac{1}{2^j}$, from \eqref{diff of psi lower bound} we have
\begin{equation*}
(1-\hat{\sigma})^{d-1}\alpha_+ - (1+\hat{\sigma})^{d-1}\alpha_- \leq
\max\left\{ \Theta\left(\bar \epsilon^{-2d+1}\right), \, \Theta\left(\bar \rho^{-\frac{(2d-1)(d+1)}{d}} \right) \right\} \frac{1}{2^j}.
\end{equation*}
The left hand side of the above inequality is a positive constant. Therefore,
\begin{equation*}
j \leq  \Theta \left( \max\{   \log_2(\bar{\epsilon}^{-1})  ,  \log_2(\bar{\rho}^{-1})   \} \right)
\end{equation*}
as required.
\hfill $\Box$ \vskip 0.4cm
{
	
	\begin{remark}In fact, we can quantify the constants in the proof of Theorem \ref{Thm:complexity-line-search} more explicitly, and
		obtain the exact form of the bound $ \Theta \left( \max\{   \log_2(\bar{\epsilon}^{-1})  ,  \log_2(\bar{\rho}^{-1})   \} \right)$.
		Recall that
		$$
		\bar \lambda = \max \left\{   \alpha_-^{1/d} \left[ \frac{1}{\bar \rho} (1+\hat \sigma + \frac{L_d+M}{d!} \alpha_-) \right]^{1-1/d}, \left(  \frac{\sigma^2 \alpha_-^{2/(d-1)}  }{2 \bar \epsilon}  \right)^{\frac{d-1}{d+1}   } \right\}\quad \mbox{and} \quad
		D_1 = \left(2+ \frac{2}{\sqrt{1-\sigma^2}} D\right).$$
		Introduce the following constants
		\begin{eqnarray*}
			G_1 &=& \frac{4(\bar \lambda+4 \bar{C})^2}{\hat{C}}, \\
			G_2 &=& D+ \frac{L_dD}{d!} \max (\frac{9}{4} \bar{C}, 8 \bar \lambda), \\
			G_3 &=& (1+\hat \sigma) \left[D_1 +  \frac{L_dD_1^{d+1}}{d!}  \bar \lambda \right], \\
			G_4 &=& \sum_{l=2}^d \left[  (l-1) B_l D_1 ( D_1+G_2)^{l-2} + B_{l+1} D_1 (D_1+ G_2)^{l-1}  \right] + B_2D_1,
		\end{eqnarray*}
		where
		$$
		\hat{C} = \frac{d! \sigma_l}{(L_d+M) D_1^{d-1}} ,\quad  \bar{C} = \max \left\{   \frac{\sigma^2 D}{2(1-\sigma^2) \bar \epsilon} ,  \frac{D^{(3d-1)/(2d)} (1+\sigma)^{1/d} }{
			(1-\sigma) \alpha_-^{  \frac{d-1}{d(d-2)}  }
		}
		\left(\frac{1}{\bar \rho} \right)^{\frac{d+1}{d}}
		\right\}
		$$
		and $B_1, ... , B_d$ is a sequence defined by
		$$
		B_d = \|\nabla^d f(x_*)\|, \quad B_{l-1} = \|\nabla^{l-1} f(x_*)\| + 2D_1 B_l, \ l = 2 ,..., d.
		$$
		Then the complexity bound in Theorem \ref{lemma:upper-bound-lambda} can be explicitly expressed by
		\begin{equation}\label{Exact-Bound}
		\log \left( \frac{ d G_1 \bar{\lambda}^{d-1} + (d-1) \max(G_2,G_3)^{d-2} \bar{\lambda}^2 (2\bar{\lambda}^2 G_4 + G_1G_2 + \bar{\lambda} D_1)
		}{ (1-\hat \sigma)^{d-1} \alpha_+ - (1+\hat \sigma)^{d-1} \alpha_-
		}  \right).
		\end{equation}
		It is clear that the dependence of the resulting bound
		depends logarithmically on the parameters $L_d, D$, and input parameters $\alpha_+, \alpha_-, \sigma_u$, and polynomially on $d$. The derivation of \eqref{Exact-Bound} is skipped for the sake of succinctness. % of the paper as it is technically involved.
	\end{remark}
	
	Now, for a given $\epsilon>0$, we denote
	$$
	\bar D_{\epsilon} := \sup \{\|x-x_*\|: \exists y\in \partial F(x) \ s.t. \ \|y\| < \epsilon \}.
	$$
	Combining the bounds provided in Theorem \ref{Thm:Iteration-Complexity}, Theorem \ref{Thm:complexity-line-search} and \eqref{Exact-Bound},
	we obtain the overall iteration bound for Algorithm \ref{Alg:Opt-High-Order} in terms of the {\bf ATS} calls as follows:
	%Simplify the statement in the theorem, we first introduce the notations
	%$$
	%W_1:= \frac{(1+\hat \sigma) L_d D_1^{d+1}}{d!}, \quad W_2:= \alpha_-^{1/d} \left(
	%1+ \hat \sigma + \frac{ L_d+M}{d!} \alpha_-
	%\right)^{1-1/d},\quad
	%W_3:= \sigma^{\frac{2(d-1)}{d+1}} \alpha_-^{\frac{2}{d+1}}.
	%$$

	\begin{theorem}\label{Final-complexity-bound}
		Given $\epsilon>0$. Assume that Algorithm \ref{Alg:Opt-High-Order} is implemented with $M \le 2L_d$. Set $\bar \epsilon =  \epsilon / 2$, $\bar \rho  \le \min \left\{ \frac{\epsilon}{2 \bar D_{\epsilon}}, \ \epsilon \right\} $, and define
		%	
		%	suppose $\bar \rho$ is picked such that either one of the following two conditions is satisfied	
		%	(a) $$ \bar \rho \le \min\left\{\left(\frac{\epsilon}{4W_1W_2}\right)^{d} , \frac{\epsilon^2}{4W_1W_3} \right\}.$$
		%	
		%	(b) $$\bar \rho  \le \min \left\{ \frac{\epsilon}{2 \bar D_{\epsilon}}, \ \epsilon \right\}.$$
		%	$$\bar \epsilon \le  \frac{\epsilon}{2} \quad  \bar \rho \le {\bar D_\epsilon}^{-1} \frac{\epsilon}{2}.$$
		%	where $\bar D_{\epsilon}:= \sup \{\|x-x_*\|: \exists \ y \in \partial F(x) \ s.t. \ \|y\| \le \epsilon\}$.
		$$
		K_{\epsilon}:=  \left\lceil
		\frac{d+1}{2} \left(\frac{ 2^d}{(1-(\hat \sigma +\sigma_u)^2)^{(d-1)/2} d! \sigma_l}
		\right)^{\frac{2}{3d+1}} \left( \frac{(L_d + M) D^{d+1}}{\epsilon}
		\right)^{\frac{2}{3d+1}}
		T_{\epsilon} \right\rceil
		$$
		where
		$$
		T_{\epsilon} = \log \left( \frac{ d G_1 \bar{\lambda}^{d-1} + (d-1) \max(G_2,G_3)^{d-2} \bar{\lambda}^2 (2\bar{\lambda}^2 G_4 + G_1G_2 + \bar{\lambda} D_1)
		}{ (1-\hat \sigma)^{d-1} \alpha_+ - (1+\hat \sigma)^{d-1} \alpha_-
		}  \right).
		$$
		Then, a point $z\in \R^n$ satisfying
		$$
		F(z) - F_* \le \epsilon
		$$
		can be found by Algorithm \ref{Alg:Opt-High-Order} with no more than $K_\epsilon$ calls of {\bf ATS}.
	\end{theorem}
	
	%\begin{remark}
	%Note that we introduce two different conditions (a) and (b) for $\bar\rho$. The condition (a) guarantees that $\bar \rho$ is polynomially dependent on $\epsilon$, and thus $T_{\epsilon}$ is at the logrithm level of $1/\epsilon$. However, we also give the alternative condition (b), which matches the intuition on the relation between small subgradient value and the objective value gap. The condition (b) might allow larger $\bar \rho$ if the local condition around $x_*$ is good enough that $\bar D_{\epsilon}$ is not too large. Besides, the conditions on $\bar \epsilon$ and $\bar \rho$ are only used in the final stages of the algorithm to decide the point is close to optimality. In the implementation, it is reasonable to set a lower precision at the beginning stage of the algorithm.
	%\end{remark}
	%
	%\textbf{Proof of Theorem \ref {Final-complexity-bound}:}
	
	\begin{proof}
		We consider two cases separately. In the first case, Algorithm \ref{Alg:Opt-High-Order} terminates
		because we find a $k\le K_{\epsilon}$ such that $\|v_{k}\| \le \bar \rho$ and $\|\epsilon_{k}\|\le \bar \epsilon$. As $v_{k} = \nabla f(y_{k}) - \nabla f_{x_k}(y_{k}) + u_{k}$ and $u_{k}\in \nabla f_{x_k}(y_{k}) + \partial_{\epsilon_{k}} h(y_{k})$, we have $v_{k}\in \nabla f(y_{k}) +\partial_{\epsilon_{k}} h(y_{k}) $. Let $x_*$ be the projection of $x_0$ onto $X_*$.
		By the convexity of $f$ and $h$,
		\begin{eqnarray*}
			f(x_*) &\ge& f(y_{k}) + \langle \nabla f(y_{k}) , x_* - y_{k} \rangle \\
			h(x_*) &\ge& h(y_{k}) + \langle v_{k} - \nabla f(y_{k}),  x_* - y_{k} \rangle - \epsilon_{k} .
		\end{eqnarray*}
		Summing up the two inequalities above yields
		$$
		F_* \ge F(y_{k}) + \langle v_{k} , x_* - y_{k} \rangle - \epsilon_{k}
		\ge F(y_{k}) - \bar \rho \|y_{k}-x_*\| - \bar \epsilon .
		$$
		By the construction of $\bar{\rho}$, we have that $\|v_{k}\| \le \bar \rho \le \epsilon$. Together with the definition of $\bar D_{\epsilon}$, this implies that $\|y_{k} - x_*\| \le \bar D_{\epsilon}$. Again, by evoking the construction of $\bar{\rho}$ and $\bar{\epsilon}$, it holds
		$$
		F_* \ge F(y_{k}) - \frac{\epsilon}{2} - \frac{\epsilon}{2}  = F(y_{k}) - \epsilon.
		$$

		In the other case, condition \eqref{two-side condition} holds for every $k\le K_{\epsilon}$. Then, according to Theorem \ref{Thm:Iteration-Complexity}, for all $k \le K_{\epsilon}$ we have
		\begin{equation*}
		F(y_{k})-F_{*} \leq  \left( \frac{d+1}{2} \right)^{\frac{3d+1}{2}} \frac{2^d}{ (1-(\hat \sigma + \sigma_u)^2)^{\frac{d-1}{2}}d!\sigma_l } D^{d+1}(L_d+M)  k^{-\frac{3d+1}{2}}.
		\end{equation*}
		By the definition of $K_{\epsilon}$ and letting $k = K_{\epsilon}$ in the above inequality, one has
		$$
		F(y_{K_{\epsilon}}) - F_* \le \epsilon.
		$$
	\end{proof}
	Note that the implementation of our framework is based on the assumption that the ATS can be efficiently computed.
	By the construction, the objective in \eqref{inexact proximal} has a $\frac{1}{\lambda}$-strongly convex smooth part, which can be solved by many start-of-the-art optimization algorithms. However, there is few efficient algorithms customized for problem \eqref{inexact proximal}.
	Without further knowledge of the problem structure, %we cannot guarantee that
	the proposed approach is not necessarily more efficient as compared to, e.g., a direct application of a general-purpose convex optimization algorithm on \eqref{Prob:main}. At the end of the paper, we shall briefly discuss a method for the subroutine in the special case $d=3$ introduced by Nesterov \cite{Nesterov-2018}, which opens the door for this line of research. Of course, how to efficiently solve  %implementation of general case of
	{\bf ATS} in general remains %is not easy for now, and could be
	a further research topic.
}

\section{Concluding Remark}  { To conclude this paper, we shall discuss how to compute {\bf ATS} efficiently with $d=3$.}
Note that in STEP 2.a of Algorithm \ref{alg:bisection}, an
{ Approximate Tensor Subroutine {\bf (ATS)}
}
%approximate tensor proximal {\bf (ATP)} mapping
is required, which can be implemented in polynomial time in the case of convex optimization. In some applications, {\bf ATS} may be implemented efficiently if some additional structures on the tensor (Taylor) expansion and/or the $h$ function exist. In this subsection, we show how { \bf ATS} (i.e., solve problem  \eqref{inexact proximal}) may be computed efficiently in the absence of the non-smooth part, i.e.\ $F(x)=f(x)$, when $d=3$.
{
	Note that since $h=0$, the $\epsilon_{\beta}$ in the bisection subroutine may be simply set to 0.}

In this case, the objective function in \eqref{inexact proximal} becomes: $f_x(y)+\frac{1}{\lambda} \|y-x\|^2  =f(x)+ \Omega(y-x) $ where
\begin{equation*}\label{regulized taylor approximate function}
\Omega(z)=z^{\top}\nabla f(x) +\frac{1}{2} z^{\top}\left(\nabla^2 f(x)+\frac{1}{\lambda}I\right)z +\frac{1}{3!}\nabla^3 f(x)[z]^3 +\frac{M}{4!} \|z\|^4.
\end{equation*}
Therefore, the subproblem  \eqref{inexact proximal} is equivalent to $\min_{z \in \R^n} \Omega(z)$. {Let $M=3\kappa^2 L_3$ with $\kappa>1$. Then, a similar argument as in Lemma~4 of \cite{Nesterov-2018} implies that function $\Omega(z)$ satisfies the strong relative smoothness condition
	\begin{equation}\label{relative smooth}
	\nabla^2 \rho(z) \preceq \nabla^2\Omega(z) \preceq \frac{\kappa+1}{\kappa-1}\nabla^2 \rho(z)
	\end{equation}
	with respect to function
	$$\rho(z) = z^{\top}\left( \frac{\kappa - 1}{2\kappa} \nabla^2 f(x) + \frac{\kappa-1}{2\lambda(\kappa +1)}{I} \right) z + \frac{M-3\kappa L_3}{6} \| z \|^4.$$}
Such condition allows to minimize $\Omega(z)$ efficiently by a gradient method described in \cite{Lu-relatively smooth,Nesterov-2018}, where we need to solve the following problem in every iteration:
$$
\min\limits_{z \in \R^n} \left( a^{\top}z  +\frac{1}{2} z^{\top}Az +\frac{\gamma}{4}\|z\|^4 \right),\; A\succeq0,\;\gamma>0,
$$
which was considered at the end of Section 5 in \cite{Nesterov-2018}. According to a min-max argument in \cite{Nesterov-2018}, the above problem is shown to be equivalent to
$$
\min\limits_{\tau>0} \left(\gamma\tau^2+\frac{1}{2}a^{\top} (\gamma \tau I +A)^{-1}a \right),
$$
which is actually a univariate optimization problem with a strongly convex and analytic objective function, hence is easily solvable in practice.
{Note that a matrix inverse operation is required in the univariate optimization above. Since in all iterations of the gradient method described in  \cite{Lu-relatively smooth,Nesterov-2018}, the matrix $A$ is exactly $\nabla^2 f(x)$ throughout and only the vector $a$ varies, the matrix inverse operation needs to be performed only once. As the gradient method in \cite{Lu-relatively smooth,Nesterov-2018} is linearly convergent, the total computational cost of an {\bf ATS} call is in the order of $O(n^3 + n^2 \log(\bar \rho^{-1}))$.	
	%In implementation, it is reasonable to compute a eigenvalue factorization of matrix $A$. With the factorization, the univariate optimization problem here is readily solved. Since in all iterations of the gradient method described in  \cite{Lu-relatively smooth,Nesterov-2018}, the matrix $A$ is the same matrix $\nabla^2 f(x)$ and only vector $a$ are different, such factorization can be done just once.
}
% Samples of sectioning (and labeling) in MOOR.
% NOTE: (1) all section levels end with a period,
%       (2) capitalization is as shown (sentence style, not title style).
%
%\section{Introduction.}\label{intro} %%1.
%\subsection{Duality and the classical EOQ problem.}\label{class-EOQ} %% 1.1.
%\subsection{Outline.}\label{outline1} %% 1.2.
%\subsubsection{Cyclic schedules for the general deterministic SMDP.}
%  \label{cyclic-schedules} %% 1.2.1
%\section{Problem description.}\label{problemdescription} %% 2.

% Text of your paper here

% Appendix here
% Options are (1) APPENDIX (with or without general title) or
%             (2) APPENDICES (if it has more than one unrelated sections)
% Outcomment the appropriate case if necessary
%
% \begin{APPENDIX}{<Title of the Appendix>}
% \end{APPENDIX}
%
%   or
%

\section*{Acknowledgments.} We would like to thank Tianyi Lin at UC Berkeley for the helpful discussions at the early stages of the project.

\bibliographystyle{plain}

\appendix

\section{Proofs of the lemmas in Section \ref{Line search}}
We first establish an uniform lower bound as well as an upper bound for the sequence $\{A_k\}$.
\begin{lemma}
	Let $D$ be the distance of $x_0$ to $X_{*}$. Suppose $\{A_k\}_{k=1}^{\ell}$ is generated from Algorithm \ref{Alg:Opt-High-Order}, and the algorithm has not stopped at iteration $\ell$. Then for any integer $ 1 \le k \le \ell$, it holds that
	\begin{equation}\label{lower-bound-Ak}
	A_k \ge \frac{d! \sigma_l}{(L_d + M )\left( \frac{2}{\sqrt{1-\sigma^2}}+2 \right)^{d-1}D^{d-1}},
	\end{equation}
	and
	\begin{equation}\label{upper-bound-Ak}
	A_k \leq \max\left\{ \frac{\sigma^2 D}{2\bar{\epsilon}(1-\sigma^2)} , \frac{D^{\frac{3d-1}{2d}}(1+\sigma)^{\frac{1}{d}}}{(1-\sigma)(\alpha_-)^{\frac{(d-1)}{d(d-2)}}} \left(\frac{1}{\bar{\rho}}\right)^{\frac{d+1}{d}}\right\}.
	\end{equation}
\end{lemma}

\begin{proof}
	We first establish the lower bound. Since $\{ A_k \}$ is monotonically increasing, it suffices to lower bound $A_1$. Recall that $A_0=0$ and $A_1 = A_0 + a_1 = \lambda_1$, and the choice of large-step \eqref{two-side condition} in Algorithm~\ref{Alg:Opt-High-Order} leads to
	$$
	\frac{d! \sigma_l}{L_d + M } \leq \lambda_{1}\|y_{1}-\tilde{x}_{0} \|^{d-1} .
	$$
	Moreover, Lemma \ref{convergence speed_general} implies that
	\begin{equation}\label{bound-x-k-star}
	\|x_k - x_*\| \leq D,\quad \mbox{and} \quad \|y_k - x_*\| \leq \Big( \frac{2}{\sqrt{1-\sigma^2}}+1\Big)D,
	\end{equation}
	where $x_*$ is the projection of $x_0$ onto the optimal solution set $X_*$. Combining the above two inequalities with the fact that $\tilde x_0 = x_0$, it follows that
	$$
	\| y_1 - \tilde x_0 \| \le \| y_1 - x_* \| + \| x_* - \tilde x_0\|  \le \Big( \frac{2}{\sqrt{1-\sigma^2}}+2\Big)D.
	$$
	Therefore,
	$$A_1 = \lambda_1 \ge \frac{d! \sigma_l}{(L_d + M )\left( \frac{2}{\sqrt{1-\sigma^2}}+2 \right)^{d-1}D^{d-1}},$$
	which is a uniform lower bound of the sequence $\{ A_k \}$.
	
	Next, we provide the upper bound.
	By invoking \eqref{ineq:appr-sol2} to $(y_k, v_k, \epsilon_k, \lambda_k, \tilde {x}_{k-1})$, it holds that
	\begin{eqnarray}\label{v part}
	\lambda_k \|v_k\| &\leq& (1+\sigma)\|y_k-\tilde{x}_{k-1}\|,\\
	2 \lambda_k \epsilon_k &\leq& \sigma^2 \|y_k-\tilde{x}_{k-1}\|^2.\label{epsilon part}
	\end{eqnarray}
	Then, combining \eqref{v part} with \eqref{Bound: sum-Ak} leads to
	\begin{equation}\label{A_k-lambdak-bound}
	A_k \lambda_k \|v_k\|^2 \leq (1+\sigma)^2 \frac{A_k}{\lambda_k}\|y_k-\tilde{x}_{k-1}\|^2 \leq (1+\sigma)^2\frac{D^2}{1-\sigma^2}.
	\end{equation}
	Moreover, it follows from (\ref{sum bound C}) that
	\begin{equation*}
	\frac{A_k}{\lambda_k^{\frac{d+1}{d-1}}} \leq \sum_{j=1}^{k} \frac{A_j}{\lambda_j^{\frac{d+1}{d-1}}} \leq \frac{D^2}{(1-\sigma^2)(\alpha_-)^\frac{2}{d-2}}
	\end{equation*}
	where $\alpha_-=\frac{d!\sigma_l}{L_d+M}$.
	Combining the above two inequalities yields
	\begin{equation*}
	\left(\frac{A_k(1-\sigma^2)(\alpha_-)^{\frac{2}{d-2}}}{D^2}\right)^{\frac{d-1}{d+1}} \|v_k\|^2 A_k \leq \lambda_k \|v_k\|^2 A_k \leq \frac{1+\sigma}{1-\sigma}D,
	\end{equation*}
	or equivalently,
	\begin{equation}\label{bdd1}
	A_k \leq  \frac{D^{\frac{3d-1}{2d}}(1+\sigma)^{\frac{1}{d}}}{(1-\sigma)(\alpha_-)^{\frac{(d-1)}{d(d-2)}}}  \left(\frac{1}{\|v_k\|}\right)^{\frac{d+1}{d}}.
	\end{equation}
	On the other hand, (\ref{epsilon part}) together with \eqref{Bound: sum-Ak} implies that
	\begin{equation*}
	2A_k \epsilon_k \leq \frac{A_k}{\lambda_k} \sigma^2 \|y_k-\tilde{x}_{k-1}\|^2 \leq \frac{\sigma^2 D}{1-\sigma^2}.
	\end{equation*}
	Consequently,
	\begin{equation}\label{bdd2}
	A_k \leq \frac{\sigma^2 D}{2(1-\sigma^2)}\frac{1}{\epsilon_k}.
	\end{equation}
	%	From (\ref{bdd1}) and (\ref{bdd2}), we conclude that
	Now, since the algorithm has not been terminated, we have either $\|v_k\|\geq \bar{\rho}$ or $\epsilon_k \geq \bar{\epsilon}$, which combined with \eqref{bdd1} and \eqref{bdd2} yields that
	\begin{equation*}
	A_k \leq \max\left\{ \frac{\sigma^2 D}{2\bar{\epsilon}(1-\sigma^2)} , \frac{D^{\frac{3d-1}{2d}}(1+\sigma)^{\frac{1}{d}}}{(1-\sigma)}\left( \frac{1}{\alpha_1} \right)^{\frac{(d-1)}{d(d-2)}} \left(\frac{1}{\bar{\rho}}\right)^{\frac{d+1}{d}}\right\},
	\end{equation*}
	and the conclusion follows.
\end{proof}

Below we prove Lemma \ref{lemma:upper-bound-lambda} and Lemma \ref{lemma1} respectively.

\noindent	{\bf Proof of Lemma \ref{lemma:upper-bound-lambda}.} We first demonstrate that
\begin{equation}\label{beta-minus-bound}
\beta_- \leq \bar{\beta} = \beta(\bar \lambda).
\end{equation}
Otherwise, we have $\beta_- \ge \bar{\beta}$ and $\lambda_- > \bar{\lambda}$ as the function $\lambda(\beta)$ is strictly increasing in $\beta$. This together with Lemma \ref{M lemma 7.8} and the fact that $\lambda_- \|y_{-}-x_{-}\| \leq \alpha_-$, implies that
the vector $v_{\beta_-} = \nabla f(y_{-}) - \nabla f_{x_{-}}(y_{-})+ u_{\beta_-}$ satisfies
\begin{equation*}
v_{\beta_-} \in (\nabla f + (\partial h)^{\epsilon})(y_{-}),\ \ \ \|v_{\beta_-} \| \leq \bar{\rho},\ \ \
\epsilon \leq \bar{\epsilon},
\end{equation*}	
and thus the algorithm would have been terminated, yielding a contradiction.
So we have $\beta_- \leq \bar{\beta}$. Since
$j \ge \log_2 (2/(1-\bar{\beta}))$, we have
\begin{equation}\label{beta-plus-bound}
\beta_+ = \beta_- +\frac{1}{2^j} \leq \bar{\beta}+ \frac{1-\bar{\beta}}{2} = \frac{1+\bar{\beta}}{2}<1 .
\end{equation}
Together with the monotonicity of the function $\lambda(\beta)$ this implies that
$$
\lambda_+ \le \lambda\left(\frac{1+\bar{\beta}}{2}\right) = A_k \frac{\left((1 + \bar \beta)/{2} \right)^2}{1 - (1+\bar \beta)/2} = A_k \frac{\left( 1 + \bar \beta \right)^2}{2(1 - \bar \beta)} = \bar \lambda \frac{\left( 1 + \bar \beta \right)^2}{\bar \beta^2}.
$$
Therefore, $\lambda_+ \le A_k \frac{\left( 1 + \bar \beta \right)^2}{2(1 - \bar \beta)} \le  \left(\frac{3}{2}\right)^2A_k$ when $\bar\beta \le \frac{1}{2} $ and $\lambda_+ \le \bar \lambda \frac{\left( 1 + \bar \beta \right)^2}{\bar \beta^2} \le  \left(\frac{2}{1/2}\right)^2 \bar \lambda$ when $\bar \beta \ge \frac{1}{2} $. Combining the bounds in both cases, we have
\begin{eqnarray*}
	\lambda_+ &\le& \max\left\{ 9A_k/{4},\, 8 \bar \lambda \right\} \le \max\left\{ \Theta(\bar \epsilon^{-1}),\, \Theta\left((\bar \rho^{-1})^{\frac{d+1}{d}}\right),\, \Theta \left( ({\bar \epsilon}^{-1})^{\frac{d-1}{d+1}} \right) ,\, \Theta \left( ({\bar \rho}^{-1} )^{\frac{d-1}{d}}\right)  \right\} \nonumber \\
	&=&  \max\left\{ \Theta(\bar \epsilon^{-1}),\, \Theta\left((\bar \rho^{-1})^{\frac{d+1}{d}}\right) \right\},
\end{eqnarray*}
where the second inequality is due to the upper bounds of $\bar \lambda$ and $A_k$ in \eqref{Def-lambda-bar} and \eqref{upper-bound-Ak} respectively.
\hfill $\Box$ \vskip 0.4cm

\noindent {\bf Proof of Lemma \ref{lemma1}. }
Let $\bar v:= \bar u_+ - \nabla f_{x_+}(\bar y_+) +\nabla f_{x_-}(\bar y_+)$. Then $\bar v \in (\nabla f_{x_-} + \partial h)(\bar y_+)$. By (\ref{relation+}) and (\ref{relation-}), it holds that
\begin{eqnarray*}
	&&	\bar y_+ - \bar y_- +\lambda_+ \bar u_+ -\lambda_- \bar u_- = x_+ - x_-\\
	&\Longleftrightarrow& 	\bar	y_+ - \bar y_- + \lambda_- (\bar u_+ - \bar u_-) = (\lambda_- - \lambda_+)\bar u_+ + x_+ - x_-	\\
	&\Longleftrightarrow& \bar	y_+ - \bar y_- + \lambda_- (\bar v - \bar u_-) = \lambda_- (\bar v- \bar u_+) + (\lambda_- - \lambda_+)\bar u_+ + x_+ - x_-.
\end{eqnarray*}
Recall that $\bar v \in (\nabla f_{x_-}+\partial h)(\bar y_+)$ and $\bar u_- \in (\nabla f_{x_-}+\partial h)(\bar y_-)$,
and from the convexity of $f_{x_-} +h$ it holds that
\begin{equation*}
\langle \bar y_+ - \bar y_-, \bar v - \bar u_-\rangle \ge 0.
\end{equation*}
Therefore,
\begin{equation*}
\begin{split}
\lambda_- \|\bar v - \bar u_-\|
& \leq \|\bar y_+ - \bar y_- +\lambda_-(\bar v- \bar u_-)\| \\
& =\Big\|\lambda_- (\bar v-\bar u_+) + (\lambda_- - \lambda_+)\bar u_+ + x_+ - x_-\Big\| \\
& \leq \lambda_- \|\bar v- \bar u_+\| +\big| \lambda_- - \lambda_+ \big| \|\bar u_+\| +\|x_+ - x_-\|. \\
\end{split}
\end{equation*}
Using the previous identity and the triangle inequality of the norms implies that
\begin{equation*}
\begin{split}
\lambda_-^2 \|\bar u_+- \bar u_-\| & \le \lambda_-^2 (\|\bar u_+ - \bar v\| + \|\bar v - \bar u_-\|) \\
& \leq 2\lambda_-^2 \|\bar v- \bar u_+\| +\big| \lambda_- - \lambda_+ \big| \lambda_- \|\bar u_+\| +\lambda_-\|x_+ - x_-\| \\
& \le 2\lambda_-^2 \|\nabla f_{x_+}(\bar y_+)-f_{x_-}(\bar y_+)\| +\big| \lambda_- - \lambda_+ \big| \|\bar y_+ - x_+\| +\lambda_-\|x_+ - x_-\|.
\end{split}
\end{equation*}
\hfill $\Box$ \bigskip

Next we shall present the lemmas with proofs that were used in Section \ref{Line search}.
\begin{lemma}\label{lemma2}
	Suppose $\lambda_+$, $\lambda_-$, $\beta_+$ and $\beta_-$ are generated from Algorithm \ref{alg:bisection}.
	When the nubmer of iteration $j$ in Algorithm \ref{alg:bisection} satisfying $j \ge \log_2 (2/(1-\bar{\beta}))$ with $\bar{\beta} = \beta (\bar \lambda)$ and $\bar \lambda$ defined in \eqref{Def-lambda-bar}, we have
	\begin{equation}\label{}
	\lambda_+ - \lambda_- \leq {\left( \max \left\{ \Theta \left( {\bar \epsilon}^{-1} \right)  ,\, \Theta \left( {\bar \rho} ^{-\frac{d+1}{d}}\right) \right\} \right)^2}(\beta_+ - \beta_-).
	\end{equation}
	
\end{lemma}

\begin{proof}
	Since $j \ge \log_2 (2/(1-\bar{\beta}))$, inequality \eqref{beta-plus-bound} holds.
	By the mean-value theorem and the definition of $\lambda(\beta)$,  %\eqref{Expression-lambda-prime},
	there exists
	$\eta \in (\beta_-,\beta_+)$ such that
	\begin{equation*}
	\lambda_+ - \lambda_-  = A_k \left(\frac{1}{(1-\eta)^2}-1\right) (\beta_+ - \beta_-) \leq
	A_k \left( \frac{4}{(1-\bar{\beta})^2} -1\right) (\beta_+ - \beta_-),
	\end{equation*}
	where the inequality is due to \eqref{beta-plus-bound}.
	Recall that %We also recall Lemma \ref{Lemma:func-beta} states that
	$$
	\bar \beta = \frac{ \sqrt{\bar \lambda^2 + 4 \bar \lambda A_k}- \bar\lambda}{2 A_k} = \frac{2 \bar\lambda}{\sqrt{\bar \lambda^2 + 4 \bar \lambda A_k} + \bar\lambda}.
	$$	
	The relation of $\beta$ and $\lambda$ in \eqref{Expression-lambda} gives
	$$
	\frac{A_k}{(1 - \bar \beta)^2} = \frac{\bar \lambda^2}{A_k \bar \beta^4} = \frac{(\sqrt{\bar \lambda^2 + 4 \bar \lambda A_k} + \bar\lambda)^4}{16 \, A_k \bar \lambda^2} \le \frac{(\bar \lambda + 4  A_k)^2}{ A_k}.
	$$
	Therefore, by invoking \eqref{Def-lambda-bar}, \eqref{lower-bound-Ak} and \eqref{upper-bound-Ak}, we have
	\begin{eqnarray*}
		&& \lambda_+ - \lambda_-  \\
		&\le&  A_k \left( \frac{4}{(1-\bar{\beta})^2} -1\right) (\beta_+ - \beta_-)\\
		&\le&   \frac{4\,A_k}{(1-\bar{\beta})^2}  (\beta_+ - \beta_-)\\
		&\le & \frac{4(\bar \lambda + 4  A_k)^2}{ A_k}(\beta_+ - \beta_-)\\
		&\le & {\left( \max \left\{ \Theta \left( ({\bar \epsilon}^{-1})^{\frac{d-1}{d+1}} \right) ,\, \Theta \left( ({\bar \rho}^{-1} )^{\frac{d-1}{d}}\right) \right\} +  \max \left\{ \Theta \left( {\bar \epsilon}^{-1} \right)  ,\, \Theta \left( ({\bar \rho}^{-1} )^{\frac{d+1}{d}}\right) \right\} \right)^2}(\beta_+ - \beta_-)\\
		&\le & {\left( \max \left\{ \Theta \left( {\bar \epsilon}^{-1} \right)  ,\, \Theta \left( ({\bar \rho}^{-1} )^{\frac{d+1}{d}}\right) \right\} \right)^2}(\beta_+ - \beta_-).
	\end{eqnarray*}
\end{proof}
The following lemma is exactly Proposition 4.5 in \cite{Monteiro-Svaiter-2012}.
\begin{lemma}\label{continuity of maximal monotone operator}
	Let $A:\mathbb{R}^s\rightrightarrows\mathbb{R}^s$ be a maximal monotone operator. Then for any $x,\tilde{x} \in \mathbb{R}^s$, we have
	\begin{equation}\label{}
	\|(I+\lambda A)^{-1}(x)-(I+\lambda A)^{-1}(\tilde{x})\| \leq \|x-\tilde{x}\|.
	\end{equation}
	Moreover, if $x_* \in A^{-1}(0)$ then
	\begin{equation}\label{}
	\max\{\|(I+\lambda A)^{-1}(x)-x\|, \|(I+\lambda A)^{-1}(x)-x_*\|  \} \leq \|x-x_*\| .
	\end{equation}
\end{lemma}

Now we can bound the residual in terms of the distance between current iterate and an optimal solution.
\begin{lemma}\label{lemma:operator-bound}
	Let $T := \nabla f +\partial h$ and $T_x := \nabla f_x + \partial h$.
	Assume that $x_* \in T^{-1}(0)=(\nabla f + \partial h)^{-1} (0)$ and let $\bar{x},x \in \mathbb{R}^n$ be given. Then,
	\begin{equation} \label{ineq-1}
	\|x-(I+\lambda T_{\bar{x}})^{-1}(x)\| \leq \|x-x_*\| +\frac{\lambda (L_d + M)}{d!}\|\bar{x}-x_*\|^d .
	\end{equation}
	As a consequence, for every $x \in \mathbb{R}^n$, $x_* \in T^{-1}(0)$, and $\lambda>0$, it holds that
	\begin{equation} \label{ineq-2}
	\lambda \|x-(I+\lambda T_{\bar{x}})^{-1}(x)\|^{d-1} \leq \lambda \left( \|x-x_*\|+\frac{\lambda (L_d+M)}{d!} \|x-x_*\|^d \right)^{d-1}.
	\end{equation}
\end{lemma}

\begin{proof}
	Let $r$ be a constant mapping such that $r(x) =\nabla f(x_*) - \nabla f_{\bar{x}}(x_*)$ for any $x \in \R^n$. Then, construct $A:=T_{\bar{x}}+r$, where $A$ is also a maximal monotone operator. By Lemma \ref{continuity of maximal monotone operator},
	\begin{equation}\label{inverse-eye+A}
	\|(I+\lambda A)^{-1}(x)-x\| \leq \|x-x_*\|.
	\end{equation}
	Let $y = x + \lambda (\nabla f(x_*) - \nabla f_{\bar{x}}(x_*))$ and $z = \left( I+\lambda r +\lambda T_{\bar{x}} \right)^{-1} (y)$. We have
	$$
	x + \lambda (\nabla f(x_*) - \nabla f_{\bar{x}}(x_*))= y = \left( I+\lambda r +\lambda T_{\bar{x}} \right) (z) = z +  \lambda (\nabla f(x_*) - \nabla f_{\bar{x}}(x_*)) + \lambda T_{\bar{x}}(z).
	$$
	Canceling $\lambda (\nabla f(x_*) - \nabla f_{\bar{x}}(x_*))$ on both sides leads to
	\begin{equation*}
	(I+\lambda T_{\bar{x}})^{-1}(x) = z = \left( I+\lambda r +\lambda T_{\bar{x}} \right)^{-1} (y)=	
	(I+\lambda A)^{-1} (y).
	\end{equation*}
	Combining the above inequality with \eqref{inverse-eye+A} and Lemma \ref{Fun-Gap-Bound} we have
	\begin{equation*}
	\begin{split}
	\|x-(I+\lambda T_{\bar{x}})^{-1}(x)\| & =\|x-(I+\lambda A)^{-1} (y)\| \\
	& \leq \|x-(I+\lambda A)^{-1} (x)\|+\|(I+\lambda A)^{-1} (x)-(I+\lambda A)^{-1} (y)\| \\
	& \leq \|x-x_*\|+\lambda \|\nabla f(x_*) - \nabla f_{\bar{x}}(x_*)\| \\
	& \leq \|x-x_*\|+\frac{\lambda (L_d+M)}{d!}\|\bar{x}-x_*\|^d,
	\end{split}
	\end{equation*}
	which proves \eqref{ineq-1}, and \eqref{ineq-2} follows from \eqref{ineq-1} straightforwardly.
\end{proof}

\begin{lemma}\label{lemma:distance-x-y}
	Suppose $x_+ = x_{\beta_+}$ and $x_- = x_{\beta_-}$ are generated from Algorithm \ref{alg:bisection}, and $\bar y_+$, $\bar y_-$ are defined in \eqref{Def-bar-y}.
	When the number of iterations $j$ in Algorithm \ref{alg:bisection} satisfies $j \ge \log_2 (2/(1-\bar{\beta}))$ with $\bar{\beta} = \beta (\bar \lambda)$ and $\bar \lambda$ defined in \eqref{Def-lambda-bar}, we have
	\begin{equation*}
	\|x_+ - x_-\| \le \Big(2+\frac{2}{\sqrt{1-\sigma^2}}\Big)D (\beta_+ - \beta_-),\quad \|x_+ - \bar y_+\|\le \max\left\{ \Theta(\bar \epsilon^{-1}),\, \Theta\left(\bar \rho^{-\frac{d+1}{d}}\right) \right\}
	\end{equation*}
	and
	\begin{equation*}
	\|x_- - \bar y_-\| \le \max\left\{ \Theta\left( \bar \epsilon^{-\frac{d-1}{d+1}} \right),\, \Theta\left(\bar \rho^{-\frac{d-1}{d}}\right) \right\}.
	\end{equation*}
\end{lemma}

\begin{proof}
	Let $x_*$ be the projection of $x_0$ onto the optimal solution set $X_*$. According to Lemma \ref{convergence speed_general}, it holds that
	\begin{equation*}
	\|x_k - x_*\| \leq D \quad \mbox{and} \quad \|y_k - x_*\| \leq \Big( \frac{2}{\sqrt{1-\sigma^2}}+1\Big)D.
	\end{equation*}
	By \eqref{formula:x-beta}, we have $x_+ = (1 - \beta_+) y_k + \beta_+ x_k $ and $x_- = (1 - \beta_-) y_k + \beta_- x_k $. Therefore,
	\begin{equation}\label{x-xstar-norm-bound}
	\|x_+ - x_*\| \leq D, \quad \mbox{and} \quad \|x_- - x_*\| \leq \Big( \frac{2}{\sqrt{1-\sigma^2}}+1\Big)D,
	\end{equation}
	and
	\begin{eqnarray*}
		\| x_+ - x_- \| &=& \| (\beta_+ - \beta_-) (x_k - y_k)\| \le  (\|x_k - x_*\| + \| y_k - x_*\|)(\beta_+ - \beta_-) \\
		& \le& \Big( 2 + \frac{2}{\sqrt{1-\sigma^2}}\Big)D(\beta_+ - \beta_-).
	\end{eqnarray*}
	Recall that in the proof of Theorem \ref{Thm:complexity-line-search}, we showed that when $j \ge \log_2 (2/(1-\bar{\beta}))$, $\beta_- \le \bar \beta$ and
	inequality \eqref{bound-lambda-plus} holds.
	Applying Lemma \ref{lemma:operator-bound} with $x=\bar{x}=x_+$, inquality \eqref{bound-lambda-plus} and \eqref{inexact-exact-gap}, we have
	\begin{equation*}
	\|x_+ - y_+\| \leq (1+\hat \sigma) \|x_+ - x_*\| +  (1+\hat \sigma) \frac{\lambda_+ L_d }{d!}\|x_+ - x_*\|^{d+1}
	\leq \max\left\{ \Theta(\bar \epsilon^{-1}),\, \Theta\left( \bar \rho^{-\frac{d+1}{d}}\right) \right\}.
	\end{equation*}
	Since $\lambda(\beta)$ is monotonically increasing in $\beta$, $\beta_- \le \bar \beta$ amounts to
	$$
	\lambda_- \le \bar \lambda = \max\left\{ \Theta\left( \bar \epsilon^{-\frac{d-1}{d+1}} \right),\, \Theta\left( \bar \rho^{-\frac{d-1}{d}}\right) \right\} .
	$$
	Finally, applying Lemma \ref{lemma:operator-bound} again with $x=\bar{x}=x_-$ and using \eqref{inexact-exact-gap} yields
	\begin{equation*}
	\|x_- - y_-\| \leq  (1+\hat \sigma)  \|x_- - x_*\| +   (1+\hat \sigma)  \frac{\lambda_- L_d }{d!}\|x_- - x_*\|^{d+1}
	\leq \max\left\{ \Theta\left( \bar \epsilon^{-\frac{d-1}{d+1}} \right),\, \Theta\left( \bar \rho^{-\frac{d-1}{d}}\right) \right\}.
	\end{equation*}
	
\end{proof}

\begin{lemma}\label{lemma4}
	Suppose that $x_+ = x_{\beta_+}$ and $x_- = x_{\beta_-}$ are generated by Algorithm \ref{alg:bisection}, and $\bar y_+$, $\bar y_-$ are defined in \eqref{Def-bar-y}.
	If the number of iterations $j$ in Algorithm \ref{alg:bisection} satisfies $j \ge \log_2 (2/(1-\bar{\beta}))$ with $\bar{\beta} = \beta (\bar \lambda)$ and $\bar \lambda$ defined in \eqref{Def-lambda-bar}, then we have
	\begin{equation*}
	\|\nabla f_{x_+}(\bar y_+)-\nabla f_{x_-}(\bar y_+)\| \leq \max\left\{ \Theta\left( \bar \epsilon^{-d+1}\right), \, \Theta\left( \bar \rho^{-\frac{(d-1)(d+1)}{d}} \right) \right\} (\beta_+ - \beta_-).
	\end{equation*}
\end{lemma}	
\begin{proof}
	According to the definition of function $f_{x}(\cdot)$ in \eqref{g_x}, it holds that
	\begin{eqnarray}
	& &\|\nabla f_{x_-}(\bar y_+)- \nabla f_{x_+}(\bar y_+)\|
	\nonumber\\
	&=&\Bigg\| \sum_{\ell=1}^{d} \frac{1}{(\ell-1)!}\left(\nabla^\ell f(x_+)[\bar y_+ - x_+]^{\ell-1} - \nabla^\ell f(x_-)[\bar y_+ - x_-]^{\ell-1}  \right)
	\nonumber\\
	& &  \ + \frac{M}{d!} \left( \|\bar y_+ - x_-\|^{d-1}(y_+ - x_-) - \|\bar y_+ - x_+\|^{d-1}(\bar y_+ - x_+) \right)\Bigg\|
	\nonumber \\
	&\le&  \sum_{\ell=1}^{d} \frac{1}{(\ell-1)!}\left\|\nabla^\ell f(x_+)[\bar y_+ - x_+]^{\ell-1} - \nabla^\ell f(x_-)[\bar y_+ - x_-]^{\ell-1}  \right\| \label{gradient-bound} \\
	& & \ +  \frac{M}{d!} \left\| \|\bar y_+ - x_-\|^{d-1}(y_+ - x_-) - \|\bar y_+ - x_+\|^{d-1}(\bar y_+ - x_+) \right\| . \label{power-bound}
	\end{eqnarray}
	Note that \eqref{power-bound} can be further bounded as follows:
	\begin{eqnarray}
	& & \frac{M}{d!} \left\|  \|\bar y_+ - x_-\|^{d-1}(\bar y_+ - x_-) - \|\bar y_+ - x_+\|^{d-1}(\bar y_+ - x_+) \right\|
	\nonumber\\
	&=& \frac{M}{d!}  \left\|  (\|\bar y_+ - x_-\|^{d-1}-\|\bar y_+ - x_+\|^{d-1})(\bar y_+ - x_-)+\|\bar y_+ - x_+\|^{d-1}(x_+ - x_-) \right\|
	\nonumber\\
	&\le&  \frac{M}{d!} \Big( (d-1) \Big| \|\bar y_+ - x_-\| - \|\bar y_+ - x_+\| \Big| \max\left\{ \|\bar y_+ - x_-\|, \|\bar y_+ - x_+\| \right\}^{d-2} \|\bar y_+ - x_-\|
	\nonumber\\
	& &  \ \ \ \ \ +\|\bar y_+ - x_+\|^{d-1}\|x_+ - x_-\| \Big)
	\nonumber\\
	&\le& \frac{M}{d!} \left(  (d-1)\max\left\{ \|\bar y_+ - x_-\|, \|\bar y_+ - x_+\| \right\}^{d-2} \|\bar y_+ - x_-\|+
	\|\bar y_+ - x_+\|^{d-1} \right) \|x_+ - x_-\|
	\nonumber\\
	&\le& \frac{M}{d!} \left(  d \big( \|\bar y_+ - x_+\|+\|x_+ - x_-\| \big)^{d-1} \right) \Big(2+\frac{2}{\sqrt{1-\sigma^2}}\Big)D (\beta_+ - \beta_-)
	\nonumber\\
	&\le& \max \left\{ \Theta\big[(\bar \epsilon^{-1})^{d-1}\big] , \Theta\big[(\bar \rho^{-1})^{\frac{(d+1)(d-1)}{d}}\big] \right\}(\beta_+ - \beta_-)
	\label{power-bound2}
	\end{eqnarray}
	where the first inequality is due to \eqref{two-power-minus}, and the second last inequality is from Lemma  \ref{lemma:distance-x-y}, and that $\|\bar y_+ - x_-\| \le  \|\bar y_+ - x_+\|+\|x_+ - x_-\|$.
	
	It remains to bound \eqref{gradient-bound}. We first show by induction that for $1 \le \ell \le d$ and any convex combination of $x_-, x_+$ and $x_*$ denoted by $z$,
	\begin{eqnarray}
	\| \nabla^\ell f(z) \| &\le& \Theta(1) , \label{induction-1} \\
	\| \nabla^\ell f(x_+) - \nabla^\ell f(x_-)\| &\le& \Theta(1) (\beta_+ - \beta_-) . \label{induction-2}
	\end{eqnarray}
	Our induction works backwardly starting from the base case: $\ell = d$. Recall that $z$ is a convex combination of $x_-, x_+$ and $x_*$. By \eqref{x-xstar-norm-bound} we have
	$$\|z-x_*\| \le  \max\{\|x_+ - x_*\| , \|x_- - x_*\|\} \le \Theta(1).$$
	Therefore,
	\begin{eqnarray*}
		\|\nabla^d f(z)\| &\le& \|\nabla^d f(x_*)\| + \|\nabla^d f(x_*) - \nabla^d f(z)\| \\
		&\le& \|\nabla^d f(x_*)\| + L_d \|z-x_*\| \\
		&\le& \Theta(1).
	\end{eqnarray*}
	Moreover, by invoking \eqref{Lipschitz-continuous} and Lemma \ref{lemma:distance-x-y}, we have
	\begin{eqnarray*}
		\|\nabla^d f(x_+)-\nabla^d f(x_-)\| &\le& L_d \|x_+ - x_-\| \\
		&\le& L_d \Big(2+\frac{2}{\sqrt{1-\sigma^2}}\Big)D (\beta_+ - \beta_-)
		\le \Theta(1) (\beta_+ - \beta_-).
	\end{eqnarray*}
	Now suppose that the conclusion holds for some $\ell +1$. Consider
	$$z = t_1x_- + t_2x_+ + (1- t_1 - t_2)x_*,\; \forall \; 0 \le t_1, t_2\le 1.$$
	Denote $D_1 := \Big(2+\frac{2}{\sqrt{1-\sigma^2}}\Big)D$.
	By letting $x_t = \frac{t_1}{t_1 + t_2}x_- + \frac{t_2}{t_1 + t_2}x_+$ and \eqref{x-xstar-norm-bound}, we have $\|x_t - x_* \| \le \| x_- - x_*\| + \| x_+ - x_*\| \le  \Big(2+\frac{2}{\sqrt{1-\sigma^2}}\Big)D= D_1$.
	Consequently,
	\begin{eqnarray*}
		\| \nabla^\ell f(z) -\nabla^\ell f(x_*) \|
		&=& \| \nabla^\ell f(x_* + (t_1 + t_2)(x_t-x_*)) -\nabla^\ell f(x_*) \| \\
		&=& \left\| \int_0^{t_1 + t_2} \nabla^{\ell +1} f\left(x_* + u(x_t-x_*)\right) [x_t-x_*] du \right\| \\
		&\le&  \int_0^{t_1 + t_2} \| \nabla^{\ell +1} f\left(x_* + u(x_t-x_*)\right) \| \|x_t-x_*\| du \\
		&\le&(t_1 + t_2)D_1\Theta(1),
	\end{eqnarray*}
	where the second last inequality is due to \eqref{tensor-norm-bound} and the last inequality follows from the induction hypothesis on \eqref{induction-1}.
	Then, it follows that
	$$
	\| \nabla^{\ell} f(z) \| \le \|\nabla^\ell f(x_*) \| + (t_1 + t_2)D_1  \Theta(1) \le \Theta(1).
	$$
	Now by induction on \eqref{induction-1}, applying Lemma \ref{lemma:distance-x-y} and using \eqref{tensor-norm-bound}
	we have
	\begin{eqnarray*}
		\| \nabla^\ell f(x_+) - \nabla^\ell f(x_-) \|
		&=&\Big\| \int_{0}^1 \nabla^{\ell+1} f(x_- + t(x_+ - x_-))[x_+ - x_-]dt \Big\| \\
		&\le& \Theta(1) D_1 (\beta_+ - \beta_-) \\
		&\le&  \Theta(1) (\beta_+ - \beta_-).
	\end{eqnarray*}
	Therefore, by induction it follows that \eqref{induction-1} and \eqref{induction-2} hold for any $1 \le \ell \le d$.
	
	Now we come back to bound \eqref{gradient-bound}.
	For $2\leq \ell \leq d$,
	\begin{eqnarray}
	&& \left\|\nabla^\ell f(x_+)[\bar y_+ - x_+]^{\ell-1} - \nabla^\ell f(x_-)[\bar y_+ - x_-]^{\ell-1} \right\| \nonumber \\
	& \leq &\left\|\nabla^\ell f(x_+)[\bar y_+ - x_+]^{\ell-1} - \nabla^\ell f(x_+)[\bar y_+ - x_-]^{\ell-1} \right\| \nonumber \\
	& & + \left\|\nabla^\ell f(x_+)[\bar y_+ - x_-]^{\ell-1} - D^\ell f(x_-)[\bar y_+ - x_-]^{\ell-1} \right\| . \label{seperate into two terms}
	\end{eqnarray}
	Applying Lemma \ref{lemma:distance-x-y} and \eqref{induction-1}, the first term on the right hand side of \eqref{seperate into two terms} can be further upper bounded as follows:
	\begin{eqnarray}
	& &\left\|\nabla^\ell f(x_+)[\bar y_+ - x_+]^{\ell-1} - \nabla^\ell f(x_+)[\bar y_+ - x_-]^{\ell-1} \right\|
	\nonumber\\
	&=& \left\| \sum_{j=1}^{\ell-1} \nabla^\ell f(x_+) \left[ [\bar y_+ - x_+]^{j-1} [x_- - x_+] [\bar y_+ - x_-]^{\ell-j-1}  \right] \right\|
	\nonumber\\
	&\leq& \sum_{j=1}^{\ell-1} \|\nabla^\ell f(x_+)\| \|\bar y_+ - x_+\|^{j-1} \|\bar y_+ - x_-\|^{\ell-j-1} \|x_+ - x_-\|
	\nonumber\\
	&\le& \sum_{j=1}^{\ell-1}\Theta(1)\|\bar y_+ - x_+\|^{j-1} \Big( \|x_+ - \bar y_+\|+ \|x_- - x_+\| \Big)^{\ell-j -1} \|x_+ - x_-\|
	\nonumber\\
	&\le&(\ell-1) \Theta(1)\Big( \|x_+ - x_-\|+ \|x_+ - \bar y_+\| \Big)^{\ell-2} D_1 (\beta_+ - \beta_-)
	\nonumber\\
	&\le& \max \left\{ \Theta\big[(\bar \epsilon^{-1})^{\ell-2}\big] , \Theta\big[(\bar \rho^{-1})^{\frac{(d+1)(\ell-2)}{d}}\big] \right\}(\beta_+ - \beta_-) .
	\label{high-order-gradient-bound1}
	\end{eqnarray}
	Moreover, applying Lemma \ref{lemma:distance-x-y}, \eqref{induction-1} and \eqref{tensor-norm-bound} to the second term on the right hand side of \eqref{seperate into two terms} gives that
	\begin{eqnarray}
	& &\left\|\nabla^\ell f(x_+)[\bar y_+ - x_-]^{\ell-1} - \nabla^\ell f(x_-)[\bar y_+ - x_-]^{\ell-1}\right\|
	\nonumber\\
	&\le& \|\nabla^\ell f(x_+) - \nabla^\ell f(x_-)\| \|\bar y_+ - x_-\|^{\ell-1}
	\nonumber\\
	&\le&\Theta(1)(\beta_+ - \beta_-) \Big( \|x_+-x_-\|+ \|x_+ - \bar y_+\| \Big)^{\ell-1}
	\nonumber\\
	&\le& \max \left\{ \Theta\big[(\bar \epsilon^{-1})^{\ell-1}\big] , \Theta\big[(\bar \rho^{-1})^{\frac{(d+1)(\ell-1)}{d}}\big] \right\}(\beta_+ - \beta_-) . \label{high-order-gradient-bound2}
	\end{eqnarray}
	Putting \eqref{seperate into two terms}, \eqref{high-order-gradient-bound1} and \eqref{high-order-gradient-bound2} together yields
	\begin{eqnarray*}
		&& \left\|\nabla^\ell f(x_+)[y_+ - x_+]^{\ell-1} - \nabla^\ell f(x_-)[y_+ - x_-]^{\ell-1} \right\| \\
		&\leq&
		\max \left\{ \Theta\big[(\bar \epsilon^{-1})^{\ell-1}\big] , \Theta\big[(\bar \rho^{-1})^{\frac{(d+1)(\ell-1)}{d}}\big] \right\} (\beta_+ - \beta_-)
	\end{eqnarray*}
	for $\ell = 2,..., d$. When $\ell =1$, \eqref{induction-2} guarantees that
	$$\| \nabla f(x_+) - \nabla f(x_-)\| \le \Theta(1) (\beta_+ - \beta_-).$$
	Therefore, the quantity in \eqref{gradient-bound} can be bounded as
	\begin{eqnarray}
	& & \sum_{\ell=1}^{d}  \frac{1}{(\ell-1)!}\left\|\nabla^\ell f(x_+)[\bar y_+ - x_+]^{\ell-1} - \nabla^\ell f(x_-)[\bar y_+ - x_-]^{\ell-1}  \right\| \nonumber \\
	&\le&      \max \left\{ \Theta\big[(\bar \epsilon^{-1})^{d-1}\big] , \Theta\big[(\bar \rho^{-1})^{\frac{(d+1)(d-1)}{d}}\big] \right\} (\beta_+ - \beta_-). \label{gradient-bound2}
	\end{eqnarray}
	Finally, replacing \eqref{power-bound} and \eqref{gradient-bound} with \eqref{power-bound2} and \eqref{gradient-bound2} respectively leads to the desired conclusion.	
\end{proof}

\end{document}